\documentclass[12pt]{article}
\usepackage{amsmath, amsfonts, amsthm}
\usepackage{times}
\usepackage{geometry}
\def \VL {V_{\textrm{Left}}}
\def \WL {W_{\textrm{Left}}}
\def \VR {V_{\textrm{Right}}}
\def \WR {W_{\textrm{Right}}}

\def \RR {\mathbb R}

\def \eps {\varepsilon}
\def \vphi {\varphi}

\def \cE {\mathcal E}

\def \cT {\mathcal T}
\def \cK {\mathcal K}

\def \cI {\mathcal I}

\def \cK {\mathcal K}

\newtheorem{theorem}{Theorem}[section]
\newtheorem{lemma}[theorem]{Lemma}
\newtheorem{proposition}[theorem]{Proposition}
\newtheorem{corollary}[theorem]{Corollary}
\newtheorem{remark}[theorem]{Remark}

\newtheorem{definition}[theorem]{Definition}

\def\myffrac#1#2 in #3{\raise 2.6pt\hbox{$#3 #1$}\mkern-1.5mu\raise 0.8pt\hbox{$
#3/$}\mkern-1.1mu\lower 1.5pt\hbox{$#3 #2$}}

\def\qed{\hfill $\vcenter{\hrule height .3mm
\hbox {\vrule width .3mm height 2.1mm \kern 2mm \vrule width .3mm
height 2.1mm} \hrule height .3mm}$ \bigskip}

\begin{document}

\title{Convex geometry and waist inequalities}
\author{Bo'az Klartag}
\date{}
\maketitle

\begin{abstract}  This paper presents connections between
Gromov's work on  isoperimetry of waists and Milman's
work on the $M$-ellipsoid of a convex body. It is proven that
any convex body $K \subseteq \RR^n$ has a linear image $\tilde{K} \subseteq \RR^n$
of volume one satisfying the following waist inequality: Any continuous map $f:\tilde{K} \rightarrow \RR^{\ell}$
has a fiber  $f^{-1}(t)$ whose $(n-\ell)$-dimensional volume is at least $c^{n-\ell}$,
where $c > 0$ is a universal constant. In the specific case where $K = [0,1]^n$
it is shown that one may take $\tilde{K} = K$ and $c = 1$, confirming a conjecture by Guth. We furthermore exhibit relations between waist inequalities
and various geometric characteristics of the convex body $K$.
\end{abstract}


\section{Introduction}
\label{intro}

 The spherical waist inequality
states that any continuous function $f$ from the unit sphere $S^n = \{ x \in \RR^{n+1} \, ; \, |x| = 1  \}$
to $\RR^\ell$
has a large fiber: there exists $t \in \RR^\ell$ such that the fiber $f^{-1}(t)$ has a large $(n-\ell)$-dimensional volume,
at least as large as that of the sphere $S^{n- \ell}$.
In Gromov's paper \cite{Gr_filling}, this inequality is extracted from Almgren's work in the 1960s,
up to some mild technical assumptions on the function $f$.
A completely new proof of a spherical waist inequality
was given by Gromov in \cite{Gr}, where additionally the following Gaussian waist inequality  is proven:

\begin{theorem} \label{thm_gwaist}
Let $1 \leq \ell \leq n$ and let $f: \RR^n \rightarrow \RR^\ell$ be a continuous function. Then there exists $t \in \RR^\ell$
such that the fiber  $L = f^{-1}(t)$ satisfies
\begin{equation}  \gamma_n(L + r B^n) \geq \gamma_{\ell}( r  B^{\ell} ) \qquad \qquad \text{for all} \ r > 0. \label{eq_512} \end{equation}
\end{theorem}

Here, $\gamma_n$ is the standard Gaussian measure on $\RR^n$, i.e., its density is $x \mapsto \frac{e^{-|x|^2/2}}{(2 \pi)^{n/2}}$, while $B^{n} = \{ x \in \RR^{n} \, ; \, |x| \leq 1 \}$ and
$L + r B^n = \{ x + r y \, ; \, x \in L, y \in B^n\}$. In the case $\ell = 1$, Theorem \ref{thm_gwaist} follows from the well-known
Gaussian isoperimetric inequality. One of the standard proofs of this isoperimetric inequality employs the
convex localization method of Payne and Weinberger \cite{PW}, Gromov and Milman \cite{GM} and Kannan, Lov\'asz and Simonovits \cite{KLS, LS}.
When  $\ell \geq 2$, the proof by Gromov \cite{Gr} combines a Borsuk-Ulam type theorem
with a localization method in which the bisection procedure stops when arriving at an $\ell$-dimensional ``pancake''
rather than a $1$-dimensional ``needle''.
Further explanations and a self-contained proof of Theorem~\ref{thm_gwaist} are given below.

\medskip Our goal in this paper is to emphasize
the relevance of Gromov's technique to convex geometry. For example, it  leads to our next  theorem.
Write $AG_{n, \ell}$ for the collection of all affine $\ell$-dimensional subspaces of $\RR^n$,
the affine Grassmannian. A convex body is a compact, convex set with a non-empty interior.

\begin{theorem} Let $K \subseteq \RR^n$ be a convex body. Then for any $\ell=1,\ldots,n$ and a continuous
function $f: K \rightarrow \RR^{\ell}$,
$$ \sup_{t \in \RR^{\ell}} Vol_{n-\ell}^*(f^{-1}(t)) \cdot \sup_{E \in AG_{n, \ell}} Vol_{\ell}(K \cap E) \geq Vol_n(K). $$
\label{thm_section}
\end{theorem}

In this paper,  $Vol^*_{\ell}$ is the {\it lower Minkowski $\ell$-volume}. That is, for $A \subseteq \RR^n$ and $0 \leq \ell \leq n$ we set
$$ Vol^*_{n-\ell}(A) = \liminf_{\eps \rightarrow 0^+} \frac{Vol_n(A + \eps B^n)}{\beta_\ell \cdot \eps^\ell } $$
where $Vol_n$ is the Lebesgue measure in $\RR^n$ and
where $\beta_\ell = \frac{\pi^{\ell/2}}{\Gamma(\ell/2 +1)} = Vol_{\ell}(B^{\ell})$.
In the case where $f: K \rightarrow \RR^\ell$ is a real-analytic map, it is known
that the lower Minkowski $\ell$-volume of $f^{-1}(t)$ coincides with its
$\ell$-dimensional Hausdorff volume \cite{AFP, Krantz}.

\medskip
It was discovered by Milman \cite{Mil} that any convex body $K \subseteq \RR^n$
has a linear image $\tilde{K}$ of the same volume, called its $M$-position, with certain non-trivial properties such as a reverse Brunn-Minkowski inequality.
Building upon these ideas, we obtain  the following:

\begin{theorem}
Let $K \subseteq \RR^n$ be a convex body of volume one. Then there exists a volume-preserving linear map $T_K: \RR^n \rightarrow \RR^n$
such that $\tilde{K} = T_K(K)$ has the following property: Let $1 \leq \ell \leq n$ and let $f: \tilde{K} \rightarrow \RR^{\ell}$ be a continuous map. Then there exists
$t \in \RR^\ell$ with
\begin{equation}  Vol^*_{n-\ell}( f^{-1}(t) ) \geq c^{n-\ell} \label{eq_1512} \end{equation}
where $c > 0$ is a universal constant. \label{thm_1752}
\end{theorem}

Theorem~\ref{thm_1752}
seems new even in the case
of a centrally-symmetric body $K$ (i.e., $K = -K$) and a linear function $f$. In this case, Pisier's regular position \cite[Chapter 7]{Pis} yields a slightly weaker
estimate in which  $c$ is replaced by $c / \log (1 + n / (n-\ell))$ on the right-hand side of (\ref{eq_1512}).
In general, we do not know the optimal value of the universal constant from Theorem~\ref{thm_1752}. More interestingly,
we currently do not have a counter-example to the variant of Theorem~\ref{thm_1752} in which
we replace $c^{n-\ell}$ by $c^\ell$ on the right-hand
side of (\ref{eq_1512}). Such a variant would  imply Bourgain's hyperplane conjecture \cite{Bourgain}
as well as the isoperimetric conjecture of Kannan, Lov\'asz and Simonovits \cite{KLS}.

\medskip The map $T_K$ that we use in Theorem~\ref{thm_1752}
is defined via an optimization procedure. In fact, assuming that the barycenter of $K$ is at the origin and abbreviating $K_0 = (-K) \cap K$, the map $T_K$ satisfies
$$
\gamma_n \left ( \frac{T_K( K_0) }{Vol_n( K_0)^{1/n}}   \right ) = \sup_{T \in SL_n(\RR)} \gamma_n \left ( \frac{T( K_0) }{Vol_n( K_0)^{1/n}} \right ), $$
where $SL_n(\RR)$ is the group of volume-preserving linear maps in $\RR^n$.
The map  $T_K$ respects the symmetries of the convex set $K$. That is, for $K \subseteq \RR^n$ with barycenter at the origin, denote
$$ Symm(K) = \{ T: \RR^n \rightarrow \RR^n  \, ; \, T \textrm{ is an isometry with } T(K) = K \}. $$
Then $T_K T = T T_K$ for all $T \in Symm(K)$. For example, in the case where $K = [-1/2,1/2]^n$ is a unit cube,
the group $Symm(K)$ consists of $2^n \cdot n!$ elements, and the only volume-preserving linear map commuting
with $Symm(K)$ is the identity map. Hence $T_K$ is the identity in the case where $K = [-1/2,1/2]^n$. In this specific case
we may determine the optimal value of the constant $c$, as follows:

\begin{theorem}[``waist of the cube''] For any $1 \leq \ell \leq n$ and a continuous function $f: (0,1)^n \rightarrow \RR^\ell$ there exists $t \in \RR^\ell$ with $Vol^*_{n-\ell}(f^{-1}(t)) \geq 1$.
\label{thm_1127_}
\end{theorem}

Theorem~\ref{thm_1127_} was conjectured by Guth \cite{Guth}. In the case where $f$ is a linear function,
Theorem~\ref{thm_1127_} goes back to Vaaler \cite{Vaaler}.
The estimate of the theorem is sharp, as is demonstrated by the example where $f(x_1,\ldots,x_n) = (x_1,\ldots,x_\ell)$.
Theorem~\ref{thm_1127_} is deduced from Theorem~\ref{thm_gwaist} via a transportation trick, see page \pageref{trick} below.
 For a convex body $K \subseteq \RR^n$
and for $\ell=1,\ldots,n$ we define the $\ell$-waist of  $K$ via
$$ w_{\ell}(K) = \inf_{f: K \rightarrow \RR^{n-\ell} } \sup_{t \in \RR^{n- \ell}} \left( Vol_{\ell}^*(f^{-1}(t)) \right)^{1/\ell}, $$
where the infimum runs over all continuous maps $f: K \rightarrow \RR^{n-\ell}$. Note that $w_{\ell}$ is $1$-homogenous in $K$, in the sense that $w_{\ell}(\lambda K) = \lambda \cdot w_{\ell}(K)$ for all $\lambda > 0$.
The $\ell$-waist is translation invariant, and it is also monotone as $w_{\ell}(A) \leq w_{\ell}(B)$ when $A \subseteq B$.
Thus $w_{\ell}(K)$ depends continuously on $K$ in the space of convex bodies.
Theorem~\ref{thm_1752} states that for any convex body $K \subseteq \RR^n$
there exists a volume-preserving linear transformation $T_K : \RR^n \rightarrow \RR^n$ with
$$ w_{\ell}(T_K(K)) \geq c \cdot Vol_n(K)^{1/n} \qquad \qquad (\ell=1,\ldots,n). $$
It is also possible to relate $w_{\ell}(K)$ to other familiar geometric characteristics of $K$. For example,
when $K$ contains the origin in its interior we establish the lower bound
\begin{equation}  w_{\ell}(K) \geq  \frac{\tilde{c}}{\sqrt{n} \cdot M(K)}
\label{eq_938}
\end{equation}
where $M(K) = \int_{S^{n-1}} \| x \|_K d \sigma(x)$  and $\| x \|_K = \inf \{ \lambda > 0 \, ; \, x / \lambda \in K  \}$
while $\sigma$ is the uniform probability measure on the sphere $S^{n-1}$.
Equivalently, $M(K)$ is half of the mean width of the polar body of $K$. Of course, $\tilde{c}$ in (\ref{eq_938}) is a positive, universal constant.
When $K = [-1,1]^n$, the estimate (\ref{eq_938}) is a bit worse than the optimal estimate, but only by a factor that is logarithmic in the dimension.
One may wonder whether $w_{\ell}(K)$ may be completely described, up to a universal constant, by some geometric properties of the convex
body $K$ itself.
When $E \subseteq \RR^n$ is a $k$-dimensional linear subspace, we show that
\begin{equation}  w_{\ell}(Proj_E(K)) \geq w_{\ell}(K) \qquad \qquad \ell=1,\ldots,k, \label{eq_944_}
\end{equation}
where $Proj_E$ is the orthogonal projection operator onto the subspace $E$ in $\RR^n$.
Furthermore,  in the case where $K$ is centrally-symmetric  we obtain
\begin{equation}
w_{\ell}(K \cap E) \geq  w_{\ell}(K) / 2 \qquad \qquad \ell=1,\ldots,k.
\end{equation}
Theorem \ref{thm_1127_} and inequality (\ref{eq_944_}) yield lower bounds on the $\ell$-waist of zonotopes.
In addition to the localization technique,
the second ingredient in our proof of Theorem~\ref{thm_1752} is related to the concept
of $\psi_2$-bodies. For $1 \leq \alpha \leq 2$,  the $\psi_\alpha$-constant
of a convex body $K \subseteq \RR^n$ with barycenter at the origin is the infimum over all $A > 0$ with the following property: For any linear functional $L: \RR^n \rightarrow \RR$,
$$ \left( \frac{\int_K |L(x)|^p dx}{Vol_n(K)} \right)^{1/p} \leq A p^{1/\alpha} \frac{\int_K |L(x)| dx}{Vol_n(K)}  \qquad \text{for all} \ p > 1. $$
For example, a computation reveals that the $\psi_2$-constant of an ellipsoid is at most $C$, where $C > 0$ is a universal constant.
A well-known consequence of the Brunn-Minkowski inequality is that 
the $\psi_1$-constant of any convex body in any dimension is bounded by a universal constant
(see, e.g., \cite[Section 3.5.3]{AGM}).

\begin{theorem} Let $K \subseteq \RR^n$ be a convex body with barycenter at the origin. Then there exists a centrally-symmetric convex body $T \subseteq K$
whose $\psi_2$-constant is at most $C_1$, such that
$$ \left( \frac{Vol_n(K)}{Vol_n(T)} \right)^{1/n} < C_2. $$
Here, $C_1, C_2 > 0$ are universal constants. \label{thm_psitwo}
\end{theorem}

Theorem~\ref{thm_psitwo} is proven in Section~\ref{M_sec}
by employing the Gaussian M-position, which is a specific type
of Milman's position promoted by Bobkov \cite{bobkov}.
Theorem~\ref{thm_1752} is deduced from Theorem~\ref{thm_section}
and Theorem~\ref{thm_psitwo}.

\medskip
Sections \ref{sec2}, \ref{sec3} and \ref{sec4}
are devoted
to the multi-dimensional localization technique in convex geometry, culminating
in Theorem~\ref{prop_941}. The latter theorem is then applied in Section~\ref{merits},
where we prove Theorem \ref{thm_gwaist}, Theorem \ref{thm_section} and many of the statements stipulated above.

\medskip We write $x + A = \{ x + y \, ; \, y \in A \}$ and $\lambda A = \{ \lambda x \, ; \, x \in A \}$ for $x \in \RR^n, \lambda \in \RR$ and $A \subseteq \RR^n$.
The standard scalar product in $\RR^n$ is denoted by $\langle \cdot, \cdot \rangle$
and $|x| = \sqrt{\langle x, x \rangle}$ for $x \in \RR^n$. All of our  group actions are left group actions. We write $\overline{A}$ for the closure
of the set $A$. A smooth function is $C^{\infty}$-smooth.

\medskip \emph{Acknowledgements.} I would like to thank Semyon Alesker, Sasha Barvinok  and Lev Buhovski
for interesting  discussions. I am grateful to Itai Benjamini for bringing Guth's paper \cite{Guth} to my attention.
Supported by a grant from the European Research Council.

\section{The Borsuk-Ulam theorem and its relatives}
\label{sec2}
\setcounter{equation}{0}

Let $X$ be a compact, connected, differentiable manifold of dimension $N$. Let $G$ be a finite group acting smoothly on $X$.
An orbit of $G$ is a finite set of the form $$ \{ g.x \, ; \, g \in G \} $$ for some $x \in X$.
We assume that the $G$-action on $X$ is free: this means that any orbit is of size exactly $\#(G)$, the cardinality of the finite group $G$. Suppose that additionally we are given an
arbitrary action of $G$ by linear isometries
on $\RR^N$. A map $F: X \rightarrow \RR^N$ is called
$G$-equivariant if
$$ F(g.x) = g.F(x)  \qquad \qquad \text{for all} \ x \in X, g \in G. $$
Note that the collection of all $G$-equivariant functions forms a vector space.

\begin{theorem} Assume that there exists a $G$-equivariant, smooth function $f: X \rightarrow \RR^N$ of which $0$ is a regular
value, such that $f^{-1}(0)$ is an orbit of $G$. Then any $G$-equivariant, continuous function $h: X \rightarrow \RR^N$ has to vanish somewhere in $X$.
\label{thm_921}
\end{theorem}

\begin{proof} Assume by contradiction that $h$ never vanishes. Set $\eps = \inf |h| > 0$.
We claim that there exists a function $\vphi: [0,1] \times X \rightarrow \RR^N$
with the following properties:
\begin{enumerate}
\item[(i)] The function $\vphi$ is smooth, and $0$ is a regular value of $\vphi$.
\item[(ii)] For any $t \in [0,1]$, the function $\vphi_t(x) = \vphi(t, x)$ is $G$-equivariant.
\item[(iii)] For any $x \in X$, we have $|\vphi_1(x) - h(x)| < \eps$.
\item[(iv)] The set $\vphi_0^{-1}(0)$ is an orbit of $G$,
and $0$ is a regular value of $\vphi_0$.
\end{enumerate}
In fact, the function $(t,x) \mapsto (1 - t) f(x) + t h(x) $ already satisfies properties (ii), (iii) and (iv).
Below we
explain how to modify the latter function slightly so that property (i) would hold true as well.
Assuming for now that such a function $\vphi$ exists,
we define
$$ Z = \vphi^{-1}(0) \subseteq [0,1] \times X. $$
Thanks to property (i) and the implicit function theorem, the set $Z$
is a compact, one-dimensional, smooth manifold with boundary points, and these boundary points are contained in  $\{ 0, 1 \} \times X$.
Therefore any connected component of $Z$ is diffeomorphic either to a circle or to the interval $[0,1]$ (see, e.g., \cite{milnor}  for the classification of $1$-dimensional manifolds with boundary).
Denote $$ Z_0 = \vphi_0^{-1}(0) \subseteq X. $$ Then $\{ 0 \} \times Z_0 \subseteq Z$.
Pick a connected component $J$ of $Z$ that intersects $\{ 0 \} \times Z_0$.
Since $0$ is a regular value of
$\vphi_0$, at any point $p \in \{ 0 \} \times Z_0$ the tangent line to $Z$
at $p$ is transversal to the slice $\{ 0 \} \times X$. We conclude that  $J$ is homeomorphic to a closed interval
and that $J$ cannot intersect $\{ 0 \} \times Z_0$ at more than two points.
Let $\gamma: [0,1] \rightarrow J$ be a homeomorphism. The endpoints of $J$ lie in $\{ 0,1 \} \times X$, but the function $\vphi$ does not vanish on $\{ 1 \} \times X$, hence
\begin{equation} J \cap [ \{ 0 \}  \times  Z_0] = \{ \gamma(0), \gamma(1) \} \qquad \text{with} \ \gamma(0) \neq \gamma(1). \label{eq_853} \end{equation}
Write $\gamma(t) = (a(t), x(t)) \in [0,1] \times X$.
The set $Z_0$ is an orbit of $G$. By (\ref{eq_853}), there exists $g \in G$ which is not the identity element such that
$$ g.x(0) = x(1). $$ Since $\vphi_t$ is $G$-equivariant, necessarily $\tilde{\gamma}(t) = (a(t), g.x(t)) \in Z$
for all $t \in [0,1]$. The continuous curve $\tilde{\gamma}: [0,1] \rightarrow Z$ is one-to-one,
and $\tilde{\gamma}(0), \tilde{\gamma}(1) \in \{ 0 \} \times Z_0$ with $\tilde{\gamma}(0) = \gamma(1) \in J$.
By using (\ref{eq_853}), we conclude that  $\tilde{\gamma}: [0,1] \rightarrow J$
is a homeomorphism with  $\tilde{\gamma}(i) = \gamma(1-i)$ for $i=0,1$. In particular, the map $$ \gamma^{-1} \circ \tilde{\gamma}: [0,1] \rightarrow [0,1] $$ is a homeomorphism of $[0,1]$
that switches the points zero and one. By the mean value theorem there exists $t_0 \in [0,1]$ such that $\gamma^{-1}(\tilde{\gamma}(t_0)) = t_0$.
Therefore $\gamma(t_0) = \tilde{\gamma}(t_0)$ and
\begin{equation} g.x(t_0) = x(t_0) \qquad \text{while} \qquad x(t_0) \in X. \label{eq_1006} \end{equation}
On the other hand, $g \in G$ is not the identity element
and the $G$-action is free on $X$, in contradiction to (\ref{eq_1006}).
\end{proof}

\begin{proof}[Proof of the existence of a function $\vphi$ satisfying properties (i),...,(iv)]
A standard argument based on a smooth partition of unity shows the existence of a
smooth function $\bar{h}: X \rightarrow \RR$ with $\sup_X |h - \bar{h}| < \eps/2$ (see, e.g., \cite[Theorem 2.2]{H}). Since $h$ is $G$-equivariant, the smooth
function
$$ \tilde{h}(x) = \frac{1}{\#(G)} \sum_{g \in G}  g^{-1}.\bar{h}(g.x) $$
is $G$-equivariant and it satisfies $\sup |h - \tilde{h}| < \eps/2$. Next, for any $x \in X$ let us select two open neighborhoods $V_x, U_x \subseteq X$ with $\overline{V_x} \subseteq U_x$ such that
$g.y \not \in U_x$ for any $y \in U_x$ and for any $g \in G$ which is not the identity element.
By compactness, the open cover $\{ V_x \}_{x \in X}$
admits a finite subcover $$ V_{x_1},\ldots,V_{x_L} \subseteq X. $$
The function $f: X \rightarrow \RR^N$ has $0$ as a regular value, which it attains
at precisely $\#(G)$ points. Hence there exists a neighborhood of the function $f$ in the $C^1$-topology that consists of functions having $0$ as a regular
value, which they attain at precisely $\#(G)$ distinct points of $X$.

\medskip Let $\theta_i: X \rightarrow [0,1]$ be a smooth function
that equals one on $V_{x_i}$ and is supported on $U_{x_i}$.
Write $B^n = \{ x \in \RR^n \, ; \, |x| \leq 1 \}$.
Then for a sufficiently small $\delta > 0$,
for all choices of $\xi_1,\ldots,\xi_L \in B^n$, the following holds: The function
$$ \tilde{f}(x) = f(x) + \delta \sum_{g \in G}  \sum_{i=1}^L \theta_i(g.x) g^{-1}.\xi_i $$
is a smooth, $G$-equivariant function on $X$, such that
$0$ is a regular value of $\tilde{f}$ which is attained at exactly $\#(G)$ points.
Hence $\tilde{f}^{-1}(0)$ is an orbit of $G$.
By decreasing $\delta$ if necessary, we may assume that $\delta \cdot \#(G) \cdot L < \eps/2$. Now let $\xi_1,\ldots,\xi_L$ be independent random vectors,
distributed uniformly on $B^n$. Let us define for $t \in [0,1]$ and $x \in X$,
$$ \vphi(t,x) = (1 - t) f(x) + t \tilde{h}(x) + \delta \sum_{g \in G}  \sum_{i=1}^L \theta_i(g.x) g^{-1}.\xi_i. $$
With probability one of selecting $\xi_1,\ldots,\xi_L$, the smooth function $\vphi$ satisfies properties (ii), (iii) and (iv).
It remains to verify property (i). It suffices to show that for any fixed $i$,
with probability one, the value $0$ is a regular value of $\vphi$ in $[0,1] \times V_{x_i}$. Observe that in the set $[0,1] \times V_{x_i}$ we may decompose
\begin{equation}  \vphi = \vphi_i + \delta \xi_i, \label{eq_1626} \end{equation}
where $\vphi_i(t,x)$ is some smooth function stochastically independent of $\xi_i$. Since $\xi_i$ is uniformly distributed in $B^n$,
Sard's theorem and the representation (\ref{eq_1626}) imply that $0$ is a regular value of $\vphi|_{[0,1] \times V_{x_i}}$, with probability one
of selecting $\xi_i$. We have thus shown that (i) holds true with probability one.
\end{proof}

Theorem~\ref{thm_921} is now  proven. The above proof of Theorem~\ref{thm_921} is modeled on the homotopy proof
of the Borsuk-Ulam theorem which may be found in Matousek \cite[Section 2.2]{Mat}.

\section{The convex localization method}
\label{sec3}
\setcounter{equation}{0}

For $u \in S^n \subseteq \RR^{n+1}$ define
$$ H(u) = \left \{ x \in \RR^n \, ; \, u_{n+1} + \sum_{i=1}^n u_i x_i \geq 0 \right \}. \label{eq_220`}
$$
The set $H(u) \subseteq \RR^n$
is usually a closed halfspace, yet when $u = \pm (0,\ldots,0,1)$ it is either the empty set or the whole of $\RR^n$.
An $\ell$-dimensional subsphere in $S^n$ is the intersection
of $S^n$ with an $(\ell+1)$-dimensional linear subspace in $\RR^{n+1}$.
For example, given an affine subspace $E \subseteq \RR^n$ of dimension $n-\ell-1$, the set
$$ \left \{ u \in S^n \, ; \, E \subseteq \partial H(u) \right \}
=\left \{ u \in S^n \, ; \, \forall x \in E, \, u_{n+1} + \sum_{i=1}^n u_i x_i = 0 \right \} $$
is an $\ell$-dimensional subsphere in $S^n$.

 \medskip Let $S_1,\ldots,S_N \subseteq S^n$ be subspheres. Write $\cK_{S_1,\ldots,S_N}$
for the collection of all closed, convex sets in $\RR^n$ of the form
$$ H(u_1) \cap H(u_2) \cap \ldots \cap H(u_N) \qquad \qquad (u_1 \in S_1,\ldots,u_N \in S_N). $$
We endow $\cK_{S_1,\ldots,S_N}$ with the quotient topology from $S_1 \times S_2 \times \ldots \times S_N$.
Thus, a function $\cI$ on $\cK_{S_1, \ldots,S_N}$ is continuous
if the expression
$$ \cI \left( H(u_1) \cap H(u_2) \cap \ldots \cap H(u_N)  \right) $$
depends continuously on $u_1 \in S_1,\ldots,u_N \in S_N$. The collection of all
$K \in \cK_{S_1,\ldots,S_N}$ with a non-empty interior is an open subset of $\cK_{S_1,\ldots,S_N}$, and in this open subset the topology
is metrizable.

\medskip
We say that the sets $A_1,\ldots,A_T$ form a partition
of a measure space $\Omega$ up to measure zero, if $A_1 \cap F, \ldots, A_T \cap F$ are disjoint sets whose union is $F$ for some set $F \subseteq \Omega$ of full measure.
The convex partition theorem of Gromov \cite{Gr} is the following result:

\begin{theorem} Let $N \geq 1, 1 \leq \ell \leq n$ and let
 $S_1,\ldots,S_N \subseteq S^n$ be $\ell$-dimensional subspheres.
Assume that $\cI: \cK = \cK_{S_1,\ldots,S_N} \rightarrow \RR^\ell$ is a continuous functional.
 Then there exist convex sets
$K_1,\ldots,K_{2^N} \in \cK$ which form a partition of $\RR^n$ up to Lebesgue measure zero, such that
$ \cI(K_1) = \cI(K_2) = \ldots = \cI(K_{2^N})$. \label{thm2}
\end{theorem}

\begin{proof}
Gromov's proof relies on the idea of successive bisections, which in this context
goes back to Payne and Weinberger \cite{PW}, Gromov and Milman \cite{GM} and Kannan, Lov\'asz and Simonovits \cite{KLS, LS}.
Write $\Omega = \{ 1,\ldots, 2^{N} \}$. A subset $V \subseteq \Omega$ is {\it dyadic} if
there exist integers $i,a \geq 0$ such that $$ V = \{ 2^i a + 1, 2^i a + 2,\ldots, 2^i (a+1) \}. $$
The collection of all dyadic
subsets $V \subseteq \Omega$ is denoted by $\hat{\cT}$, and it forms a binary tree under inclusion.
We say that a node $V \in \hat{\cT}$ is of height $h$ if
$$ \#(V) = 2^{h-1}. $$
Thus the singletons are the nodes of height $1$, which are the leaves of the tree.
In addition to leaves, there are $2^N - 1$ internal nodes  in the tree, whose height is at least $2$.
Write $\cT$ for the collection of all internal nodes.
Any  node $V \in \cT$ has two children. The child of $V$ which contains the element $\min V$
is denoted by $\VL$, and the other child is denoted by $\VR$.
 We define $X$ to be the space of all maps $ x: \hat{\cT} \setminus \{ \Omega \} \rightarrow S^n $
such that
\begin{enumerate}
\item[(i)] $\displaystyle x(\VL) + x(\VR) = 0$ for every $V \in \cT$.
\item[(ii)] $\displaystyle x(V) \in S_h$ for any node $V \in \hat{\cT} \setminus \{ \Omega \}$ whose height is $h$.
\end{enumerate}
The space $X$ is a differentiable manifold, diffeomorphic to $(S^\ell)^{2^N - 1}$.
Indeed, this follows from the fact that any map $x \in X$ is determined by its values on the set $\{ \VL \, ; \, V \in \cT \}$,
and these values can be arbitrary as long as (ii) is satisfied. For $x \in X$ and $i=1,\ldots,2^N$ we denote
$$ K_i(x) = \bigcap_{i \in V} H( x(V) ) \in \cK_{S_1,\ldots,S_N}, $$
where the intersection is over all nodes $V \in \hat{\cT} \setminus \{ \Omega \}$ containing $i$.
A moment of contemplation reveals that the family of convex sets
$ K_1(x),\ldots,K_{2^N}(x) $
forms a partition of $\RR^n$, up to Lebesgue measure zero. For $V \in \cT$ and $x \in X$  set
\begin{equation}  F_V(x) = \sum_{i \in \VL} \cI( K_i(x) ) - \sum_{i \in \VR} \cI( K_i(x) ) \in \RR^\ell. \label{eq_1636}
\end{equation}
Abbreviating $F(x) = (F_V(x))_{V \in \cT}$ we obtain a continuous map $F: X \rightarrow (\RR^\ell)^{\cT} \cong (\RR^\ell)^{2^N - 1}$.
Thus $F$ is a continuous map from the smooth manifold $X$ to the space $\RR^{\dim(X)}$. In order to conclude the proof of the theorem
we need to show that
\begin{equation}
\exists x \in X\, \ \ F(x) = 0.
\label{eq_1530}
\end{equation}
Indeed, once (\ref{eq_1530}) is proven, we obtain $\sum_{i \in \VL} \cI( K_i(x) ) = \sum_{i \in \VR} \cI( K_i(x) )$ for all $V \in \cT$, and consequently,
$$ \cI(K_1(x)) = \cI(K_2(x)) = \ldots = \cI(K_{2^N}(x)). $$
For the proof of (\ref{eq_1530}) we will consider a certain group action.
With any node $W \in \cT$  of height $h$ we associate
an involution $i_W: \Omega \rightarrow \Omega$ given by
\begin{equation}  i_W(j) =  \left \{ \begin{array}{ll} j & j \not \in W \\
j + 2^{h-2} & j \in \WL \\
j - 2^{h-2} & j \in \WR \\ \end{array} \right. \label{eq_1640} \end{equation}
Note that $i_W(V) \in \hat{\cT}$ for any $V \in \hat{\cT}$.
We may therefore view $i_W$ from now on as a map  $i_W: \hat{\cT} \rightarrow \hat{\cT}$.
This map  $i_W$ switches between the two subtrees of the vertex $W$,
leaving the rest of the tree intact.
Write $G$ for the group generated by all of the involutions $i_W$ for $W \in \cT$.
Observe that any $g \in G$ which is not the identity element may be written as
\begin{equation} g = i_{W_1} i_{W_2} \ldots i_{W_m} \label{eq_1648} \end{equation}
with $m \geq 1$ while $W_p \neq W_q$ and $h(W_p) \leq h(W_q)$ for all $p < q$. Here, $h(W_p)$ is the height of $W_p$.
With the representation (\ref{eq_1648}) we have that $g(W_m) = g^{-1}(W_m) = W_m$
 and $g^{-1}(W_p) = i_{W_m} i_{W_{m-1}} \ldots i_{W_{p+1}}(W_p)$ for $p <m$.
We see that  $G$ is a group of $2^{\#(\cT)}$ elements. The group $G$ acts on $X$ by permuting coordinates. That is, we define
\begin{equation}  (g.x)(V) = x(g^{-1}(V)) \qquad \qquad (g \in G, x \in X, V \in \hat{\cT} \setminus \{ \Omega \}). \label{eq_1658} \end{equation}
This action is free: Let $g \in G$ be as in (\ref{eq_1648}) and abbreviate $W = W_{m}$. Then $(g.x)(\WL) = x(g^{-1}(\WL)) = x(\WR) = -x(\WL)$ for any $x \in X$.
Hence $g.x \neq x $ for all $x \in X$ and  $g$ has no fixed points. According to (\ref{eq_1636}), (\ref{eq_1640}) and (\ref{eq_1658}),
\begin{equation}
   F_V(i_W.x) = \left \{ \begin{array}{ll} F_{i_W(V)}(x) & V \neq W \\ -F_V(x)  & V = W
\end{array} \right. \label{eq_1704}
\end{equation}
We move on to describing the $G$-action on $(\RR^\ell)^{\cT}$.
Given $y = (y_W)_{W \in \cT} \in (\RR^\ell)^{\cT}, V \in \cT$ and $g \in G$ of the form (\ref{eq_1648}) we set
\begin{equation} (g.y)_V = \left \{
\begin{array}{ll} y_{g^{-1}(V)} & V \not \in \{ W_1,\ldots,W_m \} \\ -y_{g^{-1}(V)} & V \in \{ W_1,\ldots,W_m \} \end{array} \right. \label{eq_1701_}
\end{equation}
In other words, we permute the coordinates according to the transformation $g^{-1} \in G$, and switch the signs
of the coordinates corresponding to the involutions that $g^{-1}$ applies. This is indeed a well-defined action of $G$ on $(\RR^\ell)^{\cT}$ by linear isometries,
as may be verified routinely.
The $G$-equivariance of the function $F$ now follows from (\ref{eq_1704}) and (\ref{eq_1701_}).

\medskip In order to apply Theorem~\ref{thm_921} and deduce the desired vanishing (\ref{eq_1530}), we
need to present a certain witness: A smooth, $G$-equivariant function
$$ f: X \rightarrow (\RR^\ell)^{\cT} $$
of which $0$ is a regular value, such that $f^{-1}(0)$ is an orbit of $G$. Fix an $\ell$-dimensional subspace $E \subseteq \RR^{n+1}$
which is {\it generic} with respect to each of the $N$ subspaces of dimension $\ell+1$ spanned by the  subspheres $S_1,\ldots,S_N \subseteq S^n$.
This means that the orthogonal complement $E^{\perp}$ intersects each $\ell$-dimensional subsphere $S_i$ at exactly two antipodal points. Let us now define
$$ f_V(x) = Proj_E \left( x(\VL) \right) \qquad \qquad \qquad (x \in X, V \in \cT) $$
where $Proj_E$ is the orthogonal projection operator onto $E$ in $\RR^{n+1}$. Setting $f(x) = (f_V(x))_{V \in \cT}$
we obtain a smooth function from $X$ to $E^{\cT} \cong (\RR^\ell)^{\cT}$. This function $f$ is $G$-equivariant, as
$$ f_V(i_W.x) = \left \{ \begin{array}{ll} f_{i_W(V)}(x) & V \neq W \\ -f_V(x) & V = W
\end{array} \right. $$
The function $f$ vanishes at exactly $2^{\#(\cT)}$ points, since
$$ \# \{ x \in S_i \, ; \, Proj_E(x) = 0 \} = 2 \qquad \text{for} \ i=1,\ldots,N. $$
The tangent spaces to $S_i \subseteq \RR^n$ at these vanishing points of $Proj_E$ are necessarily transversal to $E^{\perp}$, and
consequently all points of $f^{-1}(0)$ are regular points of $f$.
Since $f^{-1}(0)$ has the cardinality of $G$, this zero set $f^{-1}(0)$ of the $G$-equivariant function $f$  must  be an orbit of $G$.
We have thus constructed a function $f$ as required in Theorem~\ref{thm_921}. Therefore  the application
of the latter theorem is legitimate, and the proof of (\ref{eq_1530}) is complete.
\end{proof}

Similarly to Memarian \cite{Mem}, we say that a convex subset $P \subseteq \RR^n$
is an {\it $(\ell, \delta)$-pancake}, for $\ell=0,\ldots,n-1$ and $\delta > 0$, if there exists an affine $\ell$-dimensional subspace $E \subseteq \RR^n$ such that
$$ P \subseteq E + \delta B^n. $$
The following proposition shows that we can make pancakes out
of the convex sets in Theorem~\ref{thm2}.

\begin{proposition} Let $0 \leq \ell \leq n-1$ and let $R, \delta > 0$. Then there exist $N \geq 1$ and $\ell$-dimensional subspheres $S_1,\ldots,S_N \subseteq S^n$
such that for any $P \in \cK_{S_1,\ldots,S_N}$ whose interior intersects $R B^n$, the convex set $P \cap R B^n$ is an $(\ell, \delta)$-pancake. \label{prop_pancake}
\end{proposition}

\begin{proof}  Let $E_1,E_2,\ldots$ be a dense sequence in the space of all affine $(n-\ell-1)$-dimensional subspaces in $\RR^n$.
For any $(\ell+1)$-dimensional  ball $D \subseteq R B^n$ of radius $\delta/(2n)$ there exists $i$ such that $E_i \cap D \neq \emptyset$. A compactness
argument shows that there exists $N \geq 1$ such that for any $(\ell+1)$-dimensional closed ball $D \subseteq R B^n$ of radius $\delta/n$,
\begin{equation} D \cap \left( \bigcup_{i=1}^N E_i \right) \neq \emptyset. \label{eq_1216} \end{equation}
For $i=1,\ldots,N$ set $$ S_i = \left \{ u \in S^n \, ; \, E_i \subseteq \partial H(u) \right \}
=\left \{ u \in S^n \, ; \, \forall x \in E_i, \, u_{n+1} + \sum_{j=1}^n u_j x_j = 0 \right \}. $$
By linear algebra, the set $S_i$ is an  $\ell$-dimensional subsphere in $S^n$. Now let
$P \in \cK_{S_1,\ldots,S_N}$ be an arbitrary convex set whose interior intersects $R B^n$. We need to show that $P \cap R B^n$ is an $(\ell, \delta)$-pancake.
By the definition of $\cK_{S_1,\ldots,S_N}$, there exist $u_1 \in S_1,\ldots,u_N \in S_N$ such that
the interior of $P$ is disjoint from
\begin{equation} \bigcup_{i=1}^N \partial H(u_i).
\label{eq_1458} \end{equation}
However, the set in (\ref{eq_1458}) contains $\bigcup_{i=1}^N E_i$, and therefore the interior of $P$ is disjoint from $\bigcup_{i=1}^N E_i$. We thus learn from (\ref{eq_1216})
that the interior of
$P \cap R B^n$ cannot contain any closed $(\ell+1)$-dimensional ball of radius $\delta / n$. Now let $\cE$ be the John ellipsoid of $P \cap R B^n$, which
is the unique closed ellipsoid of maximal volume that is contained in $P \cap R B^n$ (see, e.g., \cite{ball})). A virtue of the John ellipsoid is that
\begin{equation}  P \cap R B^n \subseteq x_0 + n (\cE - x_0) = \{ x_0 + n (x - x_0) \, ; \, x \in \cE \} \label{eq_1214} \end{equation}
where $x_0$ is the center of the ellipsoid $\cE$. Let $\lambda_1 \geq \lambda_2 \geq \ldots \geq \lambda_n > 0$ be the
lengths of the sexi-axes of the ellipsoid $\cE$. Since the interior of $\cE$ cannot contain an $(\ell+1)$-dimensional closed ball of radius $\delta / n$, necessarily
$ \lambda_{\ell+1} \leq \delta / n$.
Now let $F$ be the $\ell$-dimensional affine subspace passing through $x_0$
and containing all of the axes of $\cE$ that correspond to $\lambda_1,\ldots,\lambda_\ell$.
Note that $\cE \subseteq F + \lambda_{\ell+1} B^n$. By (\ref{eq_1214}),
$$ P \cap R B^n \subseteq x_0 + n (\cE - x_0) \subseteq F + n \lambda_{\ell+1} B^n \subseteq F + \delta B^n, $$
and $P \cap R B^n$ is an $(\ell, \delta)$-pancake. This completes the proof.
\end{proof}

Let us note a simple variant of Theorem~\ref{thm2}, in which we replace $\RR^n$ by the sphere $S^n \subseteq \RR^{n+1}$
or by the hyperbolic space $H^n \subseteq \RR^{n+1}$, defined via
$$ H^n = \left \{ (x_1,\ldots,x_{n+1}) \in \RR^{n+1} \, ; \, -x_{n+1}^2 + \sum_{i=1}^n x_i^2 = - 1, x_{n+1} > 0 \right \}. $$
Recall that $H^n$ equipped with the Riemannian metric tensor  $g = -dx_{n+1}^2 + \sum_{i=1}^n d x_i^2$
is the hyperboloid model of the hyperbolic space. For $u \in S^n \subseteq \RR^{n+1}$ define
$$ D(u) = \left \{ x \in \RR^{n+1} \, ; \, \sum_{i=1}^{n+1} u_i x_i \geq 0 \right \}.
$$
Assume that $S_1,\ldots,S_N \subseteq S^n$ are $\ell$-dimensional subspheres.
Let $M^n \subseteq \RR^{n+1}$ be either $S^n$ or $H^n$.
Write $\cK_{S_1,\ldots,S_N}(M^n)$
for the collection all subsets of $M^n$ of the form
$$ M^n \cap D(u_1) \cap D(u_2) \cap \ldots \cap D(u_N) \qquad \qquad (u_1 \in S_1,\ldots,u_N \in S_N). $$
Note that any $K \in \cK_{S_1,\ldots,S_N}(M^n)$ is a closed, geodesically-convex subset of $M^n$.
We endow $\cK_{S_1,\ldots,S_N}(M^n)$ with the quotient topology from $S_1 \times S_2 \times \ldots \times S_N$.
Write $\mu_n$ for the Riemannian volume measure in $M^n$.
By repeating the above proof with the most straightforward modifications, we deduce the following:

\begin{theorem} Let $N \geq 1, 1 \leq \ell \leq n$ and let
 $S_1,\ldots,S_N \subseteq S^n$ be $\ell$-dimensional subspheres. Let $M^n$ be either $S^n$ or $H^n$.
Assume that $\cI: \cK = \cK_{S_1,\ldots,S_N}(M^n) \rightarrow \RR^\ell$ is a continuous functional.
 Then there exist
$K_1,\ldots,K_{2^N} \in \cK$ forming a partition of $M^n$ up to $\mu_n$-measure zero, with
$ \cI(K_1) = \cI(K_2) = \ldots = \cI(K_{2^N})$. \label{thm3}
\end{theorem}

\section{Densities with a peak  at each subspace}
\label{sec4}
\setcounter{equation}{0}

A function $\vphi: E \rightarrow [0, +\infty)$ is {\it log-concave} if $E \subseteq \RR^n$ is an affine subspace and for any $x, y \in E$ and $0 < \lambda < 1$,
\begin{equation} \vphi(\lambda x + (1-\lambda) y) \geq \vphi(x)^{\lambda} \vphi(y)^{1-\lambda}. \label{eq_929} \end{equation}
A measure $\mu$ on $\RR^n$ is log-concave if it is supported in some affine subspace $E \subseteq \RR^n$
and it has a log-concave density in $E$.
Throughout this section, we fix a convex body $V \subseteq \RR^n$ with the origin in its interior,
a dimension $\ell=0,\ldots,n$ and a continuous function $I: [0, \infty) \rightarrow [0, 1]$.

\begin{definition}[``peak property'']
Let $\vphi: \RR^n \rightarrow [0, \infty)$ be Borel measurable. We say that
$\vphi$ has the $(V, \ell, I)$-peak property if the following holds: For any affine $\ell$-dimensional subspace $E \subseteq \RR^n$ and any log-concave function $\psi: E \rightarrow [0, +\infty)$
with $\int_E{\vphi \psi} = 1$, there exists $x_0 \in E$ for which
$$ \int_{E \cap (x_0 + rV)} \vphi \psi  \geq I(r)
\qquad \qquad \text{for all} \ r \geq 0. $$
 Here, the integrals are with respect to the Lebesgue measure in $E$.
\label{def_1121}\end{definition}

In other words, we require a lower bound
on the measure of dilations of $V$ in each affine $\ell$-dimensional subspace,
with respect to any probability density that is more log-concave than $\vphi$.
Examples of probability densities with peak properties will be given in the
next section.

\begin{theorem} Let $V \subseteq \RR^n$ be a convex body with the origin in its interior, let $0 \leq \ell \leq n$ and let $I: [0, \infty) \rightarrow [0, 1]$ be continuous. Let $\mu$ be a probability measure on $\RR^n$
whose log-concave density $\vphi$ has the $(V, \ell, I)$-peak property. Then
for any continuous function $f: \RR^n \rightarrow \RR^\ell$ there exists $t \in \RR^\ell$
such that
$$  \mu( f^{-1}(t) + r V) \geq I(r) \qquad \qquad \text{for all} \ r > 0. $$
\label{prop_941}
\end{theorem}

The remainder of this section is devoted to the proof of Theorem~\ref{prop_941} in the case $1 \leq \ell \leq n-1$, as the case $\ell=n$ follows immediately from
Definition~\ref{def_1121} and the case $\ell=0$ is trivial.
We shall need  the following proposition:

\begin{proposition} Let $V, \ell, I, \mu$ be as in Theorem~\ref{prop_941}.
Let $f: \RR^n \rightarrow \RR^{\ell}$ be a bounded, continuous function and let $0 < \eps < 1/2$.
Then there exists $t \in \RR^\ell$
with
$$  \mu( f^{-1}(t) + r V) \geq I(r - \eps) - 4 \eps
\qquad \text{for all} \ r \in (\eps, 1/\eps).
$$
 \label{prop_941_}
\end{proposition}

 We now turn to the proof of Proposition~\ref{prop_941_} which requires several lemmas.
Let us fix a probability measure $\mu$ on $\RR^n$ and its log-concave density $\vphi$
which satisfies the  $(V, \ell, I)$-peak property. Assume that $1 \leq \ell \leq n-1$
and fix $0 < \eps < 1/2$. Recall that $V$ contains an open neighborhood of the origin,
that $B(x, \delta) \subseteq \RR^n$ is the closed Euclidean ball of radius $\delta$
around $x$, and that $B^n = B(0,1)$.

 \begin{lemma} \label{lem_841} There exists $\delta \in (0, \eps)$
 such that $(\delta / \eps) B^n \subseteq V$ and such that the two functions
 \begin{equation}  \vphi_{\delta}(x) = \inf_{y \in B(x, \delta)} \vphi(y) \qquad \text{and} \qquad \vphi^{\delta}(x) = \sup_{y \in B(x, \delta)} \vphi(y) \label{eq_626}
 \end{equation}
  satisfy
\begin{equation}   \int_A \vphi \leq \int_A \vphi^{\delta} \leq \eps + \int_A \vphi_{\delta} \leq \eps + \int_A \vphi
 \qquad \textrm{for any Borel set } A \subseteq \RR^n. \label{eq_636} \end{equation}  \end{lemma}

 \begin{proof} Denote $\Omega = \{ x \in \RR^n \, ; \, \vphi(x) > 0 \}$. Then $\Omega$
 is a convex set, hence its boundary $\partial \Omega$ is of zero Lebesgue measure. The function $\vphi$ is continuous in $\RR^n \setminus \partial \Omega$,
 as it is log-concave (e.g.,  \cite[Theorem 10.1]{roc}).
 We conclude that $\vphi_{\delta}$ and $\vphi^{\delta}$ from (\ref{eq_626}) satisfy that
 for almost any $x \in \RR^n$,
 $$ \lim_{\delta \rightarrow 0^+} \vphi_{\delta}(x) =
 \lim_{\delta \rightarrow 0^+} \vphi^{\delta}(x) = \vphi(x). $$
  Since $\vphi$ is a log-concave probability density,
 there exist $\alpha, \beta > 0$ such that $\vphi(x) \leq \alpha e^{-\beta |x|}$ for all $x \in \RR^n$ (e.g., \cite[Lemma 2.2.1]{Gian}). Hence
 $\vphi^{\delta}(x)$ and $\vphi_{\delta}(x)$ are bounded by $\alpha e^{\beta -\beta |x|}$, assuming that $0 < \delta < 1$. We may thus use the dominated convergence theorem, and conclude that
 $$ \lim_{\delta \rightarrow 0^+} \int_{\RR^n} \vphi_{\delta} =
 \lim_{\delta \rightarrow 0^+} \int_{\RR^n} \vphi^{\delta} = \int_{\RR^n} \vphi = 1. $$
In particular, there exists $\delta \in (0, \eps)$ such that $ \int_{\RR^n} (\vphi^{\delta} - \vphi_{\delta}) \leq \eps$.
Since $\vphi_{\delta} \leq \vphi \leq \vphi^{\delta}$, then $\int_A (\vphi^{\delta} - \vphi_{\delta}) \leq \eps$
for any $A \subseteq \RR^n$, completing the proof of (\ref{eq_636}). We may certainly assume that $\delta > 0$ is small enough
so that $\delta < \eps$ and $(\delta / \eps) B^n \subseteq V$.
 \end{proof}

 Fix $\delta > 0$ as in Lemma~\ref{lem_841} and let $\vphi_{\delta}, \vphi^{\delta}: \RR^n \rightarrow [0, \infty)$
 be defined as in (\ref{eq_626}). These two functions are log-concave, as may be verified directly from the
 definition (\ref{eq_929}).  Write $\mu_{\delta}$ for the measure on $\RR^n$ whose density is $\vphi_{\delta}$ and similarly $\mu^\delta$
 is the measure with density $\vphi^{\delta}$.
 Both measures $\mu_{\delta}$ and $\mu^{\delta}$ are finite log-concave measures. By definition,  for any $x, y \in \RR^n$,
 \begin{equation}
 |x-y| \leq \delta \qquad \Longrightarrow \qquad \vphi_{\delta}(x) \leq \vphi(y) \leq \vphi^{\delta}(x). \label{eq_843}
 \end{equation}
 Given a measurable set $P \subseteq \RR^n$ we write $C(P) \subseteq \RR^n$ for the collection
of all points $x \in \RR^n$ such that
\begin{equation}  \mu^{\delta} \left(P \cap  (x + r V) \right) \geq  \left[ I(r - \eps) - \eps  \right] \cdot \mu_{\delta}(P) \qquad \qquad \text{for} \ \ \eps \leq r \leq 1/\eps. \label{eq_1042}
\end{equation}
We think of  $C(P)$ as an approximate center of the set $P$.

\begin{lemma} Let  $P \subseteq \RR^n$ be a convex body with $\mu_{\delta}(P) > 0$. Assume that $P$ is an $(\ell, \delta)$-pancake. Then $C(P) \subseteq \RR^n$ is a closed, convex set with a non-empty interior. Moreover, for any point $x$ in the interior of $C(P)$,
\begin{equation}  \min_{r \in [\eps, 1/\eps]} \left \{ \mu^{\delta} \left(P \cap  (x + r V) \right) - \left[ I(r - \eps) - \eps \right] \cdot \mu_{\delta}(P) \right \} > 0. \label{eq_930} \end{equation} \label{lem_1704}
\end{lemma}

\begin{proof} Write $\nu_{\delta}$ and $\nu^{\delta}$ for the restriction of $\mu_{\delta}$ and $\mu^{\delta}$ to the convex
set $P$, respectively. These are two finite, log-concave measures. Since $V$ is convex, by the Pr\'ekopa-Leindler inequality
(e.g., \cite[page 3]{Pis}), the function $$ x \mapsto \nu^{\delta} \left( x + r V  \right) \qquad \qquad \qquad (x \in \RR^n) $$ is
log-concave, for all $r > 0$. The latter function is also continuous, and hence the collection of all $x \in \RR^n$ for which $\nu^{\delta} \left( x + r V  \right) \geq A$
is a closed, convex set for any $r, A > 0$. By (\ref{eq_1042}) the set $C(P)$ is the intersection of a family of closed, convex sets, and consequently this set itself is closed and convex.
 We need to show
that $C(P)$ has a non-empty interior. Since $P$ is an $(\ell, \delta)$-pancake, there exists
an  affine $\ell$-dimensional subspace $E \subseteq \RR^n$ such that
\begin{equation}
P \subseteq E + \delta B^n. \label{eq_1544}
\end{equation}
Write  $E^{\perp} = \{ x \in \RR^n \, ; \, \forall y,z \in E, \langle x, y-z \rangle = 0 \}$ for  the orthogonal complement to the affine subspace $E$.
Note that $|x - Proj_E x| \leq \delta$
for any $x \in P$, where $Proj_E$ is the orthogonal projection onto the affine
subspace $E$. Hence, from (\ref{eq_843}),
\begin{equation}  \vphi_{\delta}(x) \leq \vphi(Proj_E(x)) \leq
\vphi^{\delta}(x)
\qquad \text{for all} \ x \in P. \label{eq_1511} \end{equation}
Let $\lambda$ be the restriction of the Lebesgue measure
to the convex body $P$. Write
$\eta = (Proj_E)_* \lambda$ for the push-forward of $\lambda$ under $Proj_E$.
Then $\eta$ is an absolutely-continuous measure in the affine subspace $E$.
It follows from (\ref{eq_1511}) that for any Borel set $A \subseteq E$,
\begin{equation}
\nu^{\delta} (A + E^{\perp} ) = \int_{A + E^{\perp}} \vphi^{\delta}(x) d \lambda(x)
\geq \int_{A + E^{\perp}} \vphi(Proj_E x) d \lambda(x)
= \int_A \vphi d \eta.
 \label{eq_1531}
\end{equation}
In particular, $\int_E \vphi d \eta < \infty$. Additionally,
\begin{equation}
\int_E \vphi d \eta = \int_P \vphi(Proj_E x) d \lambda(x) \geq \int_P \vphi_{\delta}(x) d \lambda(x) = \mu_{\delta}(P) > 0.
 \label{eq_1531_}
\end{equation}
By the Brunn-Minkowski inequality, the measure $\eta$ has a log-concave density $\psi$ in the subspace $E$. The function $\vphi$ has the $(V, \ell, I)$-peak property, hence  for a certain  point $x_0 \in E$,
\begin{equation}
\forall r \geq 0, \qquad \int_{E \cap (x_0 + rV)} \vphi d \eta  = \int_{E \cap (x_0 + rV)} \vphi \psi \geq I(r) \cdot \int_E \vphi d \eta. \label{eq_1535}
\end{equation}
We shall use (\ref{eq_1531}) with $A = E \cap (x_0 + rV)$. It follows from (\ref{eq_1531}), (\ref{eq_1531_}) and
(\ref{eq_1535}) that for all $r \geq 0$,
\begin{equation} \nu^{\delta} \left \{ (E \cap (x_0 + rV) ) + E^{\perp} \right \} \geq
\int_{E \cap (x_0 + rV)} \vphi d \eta \geq I(r)  \int_E \vphi d \eta \geq I(r)  \mu_{\delta}(P).
\label{eq_1639} \end{equation}
Consider a point $x \in P \cap \left[ (E \cap (x_0 + rV) ) + E^{\perp} \right]$. Then $|x - Proj_E(x)| \leq \delta$ by (\ref{eq_1544}) while $Proj_E(x) \in x_0 + r V$.
Hence $x \in x_0 + r V + \delta B^n \subseteq x_0 + (r + \eps) V$ where we used  Lemma~\ref{lem_841} and the convexity of $V$. We have thus shown that
\begin{equation}  P \cap \left[ (E \cap (x_0 + rV) ) + E^{\perp} \right] \subseteq P \cap [x_0 + (r + \eps) V]. \label{eq_1216_} \end{equation}
The measure $\nu^{\delta}$ is supported in $P$. Thanks to (\ref{eq_1216_}) we may upgrade (\ref{eq_1639}) to the following statement: for any $r \geq 0$,
\begin{equation}  \nu^{\delta} \left \{  x_0 + (r + \eps) V \right \} \geq \nu^{\delta} \left \{ (E \cap (x_0 + rV) ) + E^{\perp} \right \}  \geq I (r) \cdot \mu_{\delta}(P). \label{eq_1821_}
\end{equation}
Consider the function
$$ h_{x}(r) = \nu^{\delta}( x + r V ) - I (r - \eps) \cdot \mu_{\delta}(P) \qquad \qquad (x \in \RR^n, r \geq \eps). $$
According to (\ref{eq_1821_}),
\begin{equation}
h_{x_0}(r) \geq 0 \qquad \text{for all} \ r \in [\eps, 1/ \eps].
\label{eq_949}
\end{equation}
The function $h_x(r)$ depends continuously on $x$ and $r$, hence $\min_{r \in [\eps, 1/\eps]} h_x(r)$ depends continuously on $x$.
We conclude from (\ref{eq_949}) that $h_x(r) \geq -\eps \mu_{\delta}(P)$ for all $r \in [\eps, 1/\eps]$ and for all $x$ in a neighbourhood of $x_0$. Consequently,
the point $x_0$ belongs to the interior of $C(P)$, which is evidently non-empty.

\medskip We move on to the proof of (\ref{eq_930}). The minimum in (\ref{eq_930}) is indeed attained by continuity.
Hence, if (\ref{eq_930}) fails, then there exists  $r \in [\eps, 1/\eps]$
and a point $x_1$ in the interior of $C(P)$ with
\begin{equation}
\nu^{\delta} \left(x_1 + r V  \right) \leq \alpha \label{eq_944}
\end{equation}
where $\alpha = \left[ I(r - \eps) - \eps \right] \cdot \mu_{\delta}(P)$.
Let us show that (\ref{eq_944}) is absurd.
The infimum of the log-concave function $x \mapsto \nu^{\delta} \left(x + r  V \right)$ in the set $C(P)$
is at least $\alpha$, by the definition of $C(P)$. By (\ref{eq_944}), this  infimum is attained at an interior point $x_1$. However, a log-concave function attaining its infimum at an interior point is constant.
Hence $\nu^{\delta} \left( x + r V \right) = \alpha$ for all $x \in C(P)$, in contradiction to (\ref{eq_949}).
\end{proof}

We would like to select a point from the center set $C(P) \subseteq \RR^n$ in a manner which is continuous in $P$,
with respect to the Hausdorff metric. The following lemma establishes the required continuity property
of the center sets.

\begin{lemma} Let $P \subseteq \RR^n$ be an $(\ell, \delta)$-pancake with $\mu_{\delta}(P) > 0$. Let $y \in C(P)$.
Let $P_1,P_2,\ldots \subseteq \RR^n$ be $(\ell, \delta)$-pancakes of positive $\mu_{\delta}$-measure,
with $P_m \longrightarrow P$ in the Hausdorff metric. Then there exist points $y_m \in C(P_m)$
with $y_m \longrightarrow y$. \label{lem_937}
\end{lemma}

\begin{proof} Otherwise, there exist $\eps_0 > 0$ and a sequence $m_1 < m_2 < \ldots$
such that $C(P_{m_i})$ is disjoint from $B(y, \eps_0)$ for all $i \geq 1$.
By Lemma~\ref{lem_1704}, there exists a point $z \in B(y, \eps_0)$ that belongs to the interior of $C(P)$. Moreover, by Lemma~\ref{lem_1704},
\begin{equation}  \min_{r \in [\eps, 1/\eps]} \left \{ \mu^{\delta} \left(P \cap  (z + r V) \right) - \left[ I(r - \eps) - \eps \right] \cdot \mu_{\delta}(P) \right \} > 0. \label{eq_1218} \end{equation}
From the Hausdorff convergence, the indicator function of $P_{m_i}$
converges to the indicator function of $P$ almost everywhere in $\RR^n$.
It follows from the dominated convergence theorem that
$\mu_{\delta}(P_{m_{i}}) \stackrel{i \rightarrow \infty} \longrightarrow \mu_{\delta}(P) $ and that for all $r > 0$,
\begin{equation}   \mu^{\delta} \left( P_{m_{i}} \cap (z + rV) \right) \stackrel{i \rightarrow \infty} \longrightarrow  \mu^{\delta} \left( P \cap (z + rV) \right). \label{eq_1219} \end{equation}
The convergence in (\ref{eq_1219}) is automatically uniform in $r \in [\eps, 1/\eps]$, because the functions
$r \mapsto \mu^{\delta}(P_{m_i} \cap (z + r V))$ and $r \mapsto \mu^{\delta}(P \cap (z + rV) )$
are continuous and non-decreasing in $r$. This uniform convergence combined with (\ref{eq_1218})
implies that the quantity
$$ \min_{r \in [\eps, 1/\eps]} \left \{ \mu^{\delta} \left(P_{m_i} \cap  (z + r V) \right) - \left[ I(r - \eps) - \eps \right] \cdot \mu_{\delta}(P_{m_i}) \right \} $$
is  positive for sufficiently large $i$. Thus $z \in C({P_{m_i}})$ for sufficiently large $i$, in contradiction.
\end{proof}

Recall that we endow the space $\cK_{S_1,\ldots,S_N}$ with the quotient topology from $S_1 \times \ldots \times S_N$.
When we discuss continuity of functions defined on subsets of $\cK_{S_1,\ldots,S_N}$, we always
refer to this quotient topology. Note that for any $R >0$, the function
$$ \cK_{S_1,\ldots,S_N} \ni P \mapsto \mu_{\delta}(P \cap R B^n) \in \RR $$ is continuous. Furthermore, suppose that $P \in \cK_{S_1,\ldots,S_N}$ satisfies that $P \cap R B^n$ has a non-empty interior.
Observe that whenever $P_1,P_2, \ldots \in \cK_{S_1,\ldots,S_N}$ satisfy $P_m \longrightarrow P$ in the quotient topology,
the sequence $P_1 \cap R B^n, P_2 \cap R B^n,\ldots$ converges to $P \cap R B^n$ in the Hausdorff metric.

\begin{proof}[Proof of Proposition~\ref{prop_941_}] Since $\int_{\RR^n} \vphi = 1$,
necessarily $\mu_{\delta}(\RR^n) \geq 1 - \eps$ according to Lemma~\ref{lem_841}.
Let us select a large number $R > 1$ such that
\begin{equation}  \mu_{\delta}(  R B^n) > 1 - (3/2) \eps > 1/4. \label{eq_958} \end{equation}
Let $S_1,\ldots,S_N \subseteq S^n$ be the $\ell$-dimensional subspheres
whose existence is guaranteed by Proposition~\ref{prop_pancake}. Let $\Omega$ be the collection of all
$P \in \cK_{S_1,\ldots,S_N}$ such that $\mu_{\delta}(P \cap R B^n) > 0$.
Then $\Omega$ is an open subset of $\cK_{S_1,\ldots,S_N}$.
For any $P \in \Omega$, the set $P \cap R B^n$ is an $(\ell, \delta)$-pancake, by Proposition ~\ref{prop_pancake}. According to Lemma~\ref{lem_1704}, the map
\begin{equation}  P \mapsto C({P \cap R B^n}) \qquad \qquad \qquad (P \in \Omega) \label{eq_930_}
\end{equation} is a well-defined map to the space of closed, convex sets in $\RR^n$ with a non-empty interior.
We would like to apply the Michael selection theorem \cite[Theorem 1.16]{BL} for the set-valued map in (\ref{eq_930_}).
Lemma~\ref{lem_937} and the remark afterwards show that the map in (\ref{eq_930_}) is lower semi-continuous on $\Omega$,
hence the application of Michael's selection theorem
is legitimate. By the conclusion of this theorem, there exists a continuous map $c: \Omega \rightarrow \RR^n$ such that
$$ c(P) \in C({P \cap R B^n}) \qquad \qquad \text{for all} \ P \in \Omega. $$
Denote $\alpha(s) = \min \{ (2^{N+3} / \eps) \cdot s, 1 \}$ for $s \geq 0$, and for $P \in \Omega$ define
$$ \cI(P) = \alpha( \mu_{\delta}(P \cap R B^n) ) \cdot f(c(P)) \in \RR^\ell. $$
For $P \in \cK_{S_1,\ldots,S_N} \setminus \Omega$ we set $\cI(P) = 0$.
The functional $\cI: \cK_{S_1,\ldots,S_N} \rightarrow \RR^\ell$ is clearly continuous at the points of $\Omega$. Since $f$ is bounded and $\alpha(0) = 0$,
this functional $\cI$ is continuous also at the points of $\cK_{S_1,\ldots,S_N} \setminus \Omega$. Theorem
\ref{thm2} thus provides $P_1,\ldots,P_{2^N} \in \cK_{S_1,\ldots,S_N}$ which form a partition of $\RR^n$ up to Lebesgue measure zero, such that
\begin{equation} \cI(P_1) = \cI(P_2) = \ldots = \cI(P_{2^N}). \label{eq_954} \end{equation}
Abbreviate $Q_i = P_i \cap R B^n$. Then $Q_1,\ldots,Q_{2^N}$
form a partition of $R B^n$ up to measure zero.
Write $A = \{ i =1,\ldots,2^N \, ; \, \mu_{\delta}(Q_i) \geq \eps / 2^{N+3} \}$.
Then  for $i \in A$ we have $P_i \in \Omega$ and $\cI(P_i) = f(c(P_i))$. From (\ref{eq_958}),
\begin{equation}  \sum_{i \in \{1,\ldots,2^N \} \setminus A} \mu_{\delta}(Q_i)
\leq 2^N \cdot \frac{\eps}{2^{N+3}} = \frac{\eps}{8} \leq \frac{\eps}{2} \cdot \mu_{\delta}(R B^n). \label{eq_1013}
\end{equation}
 Thanks to (\ref{eq_954}) there exists $t \in \RR^\ell$ such that
$t = \cI(P_i) = f(c(P_i))$ for all $i \in A$.
 We thus see that $c(P_i) \in f^{-1}(t)$ for all $i \in A$.
Since $c(P_i) \in C(Q_i)$, then for all $r \in [\eps, 1/\eps]$,
\begin{align}  \mu^{\delta} \left( f^{-1}(t) + r V \right) & \geq \nonumber
\sum_{i \in A} \mu^{\delta} \left \{ Q_i \cap  \left( f^{-1}(t) + r V \right) \right \}
\\ & \nonumber \geq \sum_{i \in A} \mu^{\delta} \left\{ Q_i \cap \left(c(P_i) + rV \right) \right \} \\  \nonumber
& \geq \sum_{i \in A} \mu_{\delta}(Q_i) \cdot [I( r-\eps) - \eps] \\ & \geq \left(1 - \frac{\eps}{2} \right) \mu_{\delta}( R B^n) \cdot \left[ I( r-\eps ) - \eps \right], \label{eq_1015} \end{align}
where we used (\ref{eq_1013}) in the last passage.
By using (\ref{eq_958}),  (\ref{eq_1015})
and Lemma~\ref{lem_841} we obtain that for all $r \in [\eps, 1/\eps]$,
\begin{align*}  \mu \left( f^{-1}(t) + r V \right) & \geq
\mu^{\delta} \left( f^{-1}(t) + r V \right) - \eps \\ & \geq
\left(1 - \frac{\eps}{2} \right) \left(1 - \frac{3\eps}{2} \right)
\left[ I( r-\eps ) - \eps \right] - \eps \geq I( r-\eps ) - 4\eps, \end{align*}
where the last inequality holds since $|I(s)| \leq 1$ for all $s \geq 0$.
\end{proof}

In order to deduce Theorem~\ref{prop_941} from Proposition~\ref{prop_941_} we need an approximation argument.

\begin{proof}[Proof of Theorem ~\ref{prop_941}]
We may assume that $f: \RR^n \rightarrow \RR^{\ell}$ is bounded, as otherwise we may replace $f$ by $T \circ f$ for some homeomorphism $T: \RR^\ell \rightarrow (-1,1)^\ell$.
The function $I$ is  continuous in $[0, \infty)$, hence
for any fixed $m \geq 1$ there exists $0 < \eps < 1/(m+1)$ with
$$ I(r - \eps) - 4 \eps \geq I(r) - 1/m \qquad \text{for}
\ r \in [1/m, m]. $$
Therefore, from Proposition~\ref{prop_941_}, for any $m \geq 1$ there exists $t_m \in \RR^\ell$
with
\begin{equation}   \mu( f^{-1}(t_m) + r V) \geq I(r) - 1/m \qquad  \text{for} \  r \in (1/m,m).
\label{eq_819} \end{equation}
We may certainly assume that $t_m$ belongs to the image of $f$.
Since $f$ is bounded, the sequence $\{ t_m \}_{m \geq 1}$ is bounded as well. Passing to a subsequence, we may
assume that $t_m \longrightarrow t$ for some $t \in \RR^\ell$. In order to conclude the proof of the theorem, it suffices to prove that for any fixed  $r, \eps >0$,
\begin{equation} \mu( f^{-1}( t) + r V)
\geq I(r) - \eps. \label{eq_914} \end{equation}
Denote $D = \sup_{x,y \in V} |x-y| < \infty$, the diameter of $V$.
Since $\mu$ is a Borel probability measure on $\RR^n$, we may find a large number $R > r D$ such that \begin{equation}
\mu \left \{  \RR^n \setminus B(0, R - r D)  \right \} \leq \eps.  \label{eq_921} \end{equation}
We claim that for
any $\delta > 0$ there exists $N \geq 1$
with
\begin{equation}
\forall m \geq N, \quad B(0, R) \cap f^{-1}(t_m) \subseteq f^{-1}(t) + \delta V.
\label{eq_825} \end{equation}
Indeed, assume by contradiction that (\ref{eq_825}) fails for any $N \geq 1$. Then there exist integers  $m_1 < m_2 < \ldots$ and points  $x_1,x_2,\ldots \in \RR^n$
with $x_k \in B(0, R) \cap f^{-1}(t_{m_{k}})$ but  $x_k \not \in f^{-1}(t) + \delta V$ for all $k$. Passing to a subsequence, we
may assume that $x_k \longrightarrow x \in B(0, R)$. Since $f$ is continuous,
$$ f(x) = \lim_{k \rightarrow \infty} f(x_k) = \lim_{k \rightarrow \infty} t_{m_{k}} = t. $$
Therefore $x_k \longrightarrow x \in f^{-1}(t)$, in contradiction to our assumption that $x_k \not \in f^{-1}(t) + \delta V$ for all $k$.
This completes the proof of (\ref{eq_825}).
From (\ref{eq_819}), (\ref{eq_921}) and (\ref{eq_825}) we obtain that for all $\delta > 0$,
\begin{align*}  \mu \left \{  \left[ f^{-1}( t) + \delta V \right] + r V \right \}
& \geq \limsup_{m \rightarrow \infty} \mu \left \{ \left[ B(0, R) \cap f^{-1}(t_m) \right] + r V \right \}
\\ & \geq - \eps + \limsup_{m \rightarrow \infty} \mu\left(f^{-1}(t_m) + r V \right)
\geq - \eps + I(r). \end{align*}
The set $f^{-1}(t) + r V$ is closed, and it equals to the set $\cap_{\delta > 0}  [f^{-1}( t) + (r + \delta) V ]$.
Since $\mu$ is a probability   measure,
$$ \mu( f^{-1}( t) + r V) = \lim_{\delta \rightarrow 0^+} \mu \left \{ \left[ f^{-1}( t) + \delta V \right] + r V \right \}
\geq I(r) - \eps. $$
We have thus established (\ref{eq_914}), and the proof is complete.
\end{proof}

\begin{remark} {\rm We conjecture that it is possible to formulate and prove the results of this section in a greater generality. For example,
the log-concavity assumptions may be replaced by $s$-concavity or by the more general Bakry-\'Emery curvature-dimension condition $CD(\kappa,N)$.
 Additionally, the Euclidean space may be replaced by a sphere or by a hyperbolic space.  }
\end{remark}

\section{The merits of convexity and log-concavity}
\label{merits}
\setcounter{equation}{0}

This section presents applications of Theorem~\ref{prop_941}.
Recall that  $\gamma_n$ is the standard Gaussian measure on $\RR^n$. For an affine
subspace $E \subseteq \RR^n$, write $\gamma_E$ for the standard Gaussian probability measure on $E$, whose
density in $E$ is proportional to the function $$ x \mapsto \exp(-|x|^2/2). $$
We say that a Borel measure $\mu$  is $1$-log-concave in $E$ if
it is supported and absolutely-continuous in $E$, and the density $d \mu / d \gamma_E$ is a log-concave function.
One property of such  measures is the following:

\begin{lemma}  Let $\nu$ be a probability measure that is $1$-log-concave in $\RR^\ell$. Then there exists $x_0 \in \RR^\ell$
such that $ \nu( x_0 + r B^{\ell}) \geq \gamma_{\ell}( r B^{\ell} )$ for all $r \geq 0$.  \label{lem_1724}
\end{lemma}

\begin{proof} Let $\vphi$ be the log-concave density of $\nu$ with respect to the Lebesgue measure on $\RR^\ell$.
Modifying $\vphi$ on a set of measure zero,
we may assume that $\vphi$ is upper semi-continuous (e.g., \cite[Section 6]{roc}).
The function $\vphi$ goes to zero exponentially fast at infinity  (e.g., \cite[Lemma 2.2.1]{Gian}). We conclude that there exists $x_0 \in \RR^\ell$ such that
\begin{equation}  \vphi(x_0) = \sup_{x \in \RR^\ell} \vphi(x). \label{eq_1043} \end{equation}
Translating, we may assume that $x_0 = 0$.
Fix a unit vector $\theta \in S^{\ell-1}$ for which $t \mapsto \vphi(t \theta)$ does not vanish identically in $(0, \infty)$.
We claim that the function $g(t) = \vphi(t \theta) e^{t^2/2}$ is non-increasing in $[0, \infty)$. Indeed, $g$ is log-concave and hence
$t \mapsto (g(t) / g(0))^{1/t}$ is non-increasing. Therefore, for any $t > 0$,
$$ \left( \frac{g(t)}{g(0)} \right)^{1/t} \leq \liminf_{s \rightarrow 0^+} \left( \frac{g(s)}{g(0)} \right)^{1/s} =
\liminf_{s \rightarrow 0^+} \left( \frac{\vphi(s \theta) }{ \vphi(0) } \right)^{1/s} \leq 1. $$
Consequently $g$ is a log-concave function on $[0, \infty)$ attaining its maximum at the origin, hence it is non-increasing.
Let $\alpha > 0$ be such that
$$ \int_0^{\infty} \left( \alpha \vphi(t \theta) - e^{-t^2/2} \right) t^{\ell-1} dt = 0. $$
Since $g(t) = \vphi(t \theta) e^{t^2/2}$ is non-increasing in $[0, +\infty)$, there exists $t_0 \in [0, +\infty]$ with
$\alpha \vphi(t \theta) \geq \exp(-t^2/2)$ if and only if $t \leq t_0$. We deduce that the function
\begin{equation} r \mapsto \int_0^r \left( \alpha \vphi(t \theta) - e^{-t^2/2} \right) t^{\ell-1} dt \qquad \qquad (r \geq 0) \label{eq_1652} \end{equation}
is non-decreasing in $[0, t_0]$ and non-increasing in $[t_0, +\infty)$. The function in (\ref{eq_1652}) vanishes at zero and infinity,
and consequently it is a non-negative function in $[0, +\infty)$. Thus, for any $r \geq 0$,
\begin{equation} \int_0^{r}  \vphi(t \theta) t^{\ell-1} dt \geq \alpha^{-1} \int_0^{r}  e^{-t^2/2} t^{\ell-1} dt = \gamma_\ell(r B^\ell) \cdot \int_0^{\infty}  \vphi(t \theta) t^{\ell-1} dt.  \label{eq_1704_}
\end{equation}
By integrating (\ref{eq_1704_}) over $\theta \in S^{\ell-1}$, we obtain $\nu( r B^\ell) \geq \gamma_\ell( r B^\ell)$, as desired.
\end{proof}

Thanks to Theorem~\ref{prop_941} we may generalize  Lemma~\ref{lem_1724} as follows:

\begin{theorem}
Let $\mu$ be a  probability measure that is $1$-log-concave in $\RR^n$ and let $0 \leq \ell \leq n$. Assume that  $f: \RR^n \rightarrow \RR^{\ell}$ is continuous. Then there exists $t \in \RR^\ell$
such that for all $r  > 0$,
\begin{equation}   \mu( f^{-1}(t) + r B^n) \geq \gamma_{\ell}( r B^{\ell} ).
\label{eq_956} \end{equation}
\label{thm_gwaist2}
\end{theorem}

\begin{proof} We apply Theorem~\ref{prop_941} with $V = B^n$
and $I(r) = \gamma_{\ell}( r B^{\ell} )$. The desired inequality (\ref{eq_956})
would follow once we verify that the log-concave density $\vphi$ of the measure $\mu$ satisfies
the $(V, \ell, I)$-peak property. Since $\mu$ is $1$-log-concave in $\RR^n$, the function
$$ x \mapsto \vphi(x) e^{|x|^2/2} $$
is log-concave. Let $E \subseteq \RR^n$ be an affine $\ell$-dimensional subspace
and let $\psi$ be a log-concave function on $E$ with $\int_E \vphi \psi = 1$.
Write $\nu$ for the probability measure on $E$ whose density in $E$ equals to  the product $\vphi \psi$.
Since $\psi$ is log-concave
and the function $ x \mapsto \vphi(x) e^{|x|^2/2} $ is log-concave as well, the probability measure $\nu$ is $1$-log-concave in the $\ell$-dimensional affine subspace $E$.
By Lemma~\ref{lem_1724}, there exists $x_0 \in E$ for which
$$ \nu( x_0 + r B^{\ell}) \geq I(r) = \gamma_{\ell}( r B^{\ell} ) \qquad \text{for all}  \ r \geq 0. $$
We have thus verified the $(V, \ell, I)$-peak property, and the proof is complete.
\end{proof}

Theorem~\ref{thm_gwaist2} implies Theorem \ref{thm_gwaist} by substituting $\mu = \gamma_n$.
Let us now see how Theorem \ref{thm_1127_} follows from Theorem~\ref{thm_gwaist}.

\begin{proof}[Proof of Theorem~\ref{thm_1127_} (``waist of the cube'')] \label{trick}
  Let $\Phi(t) = \int_{-\infty}^{t} (2 \pi)^{-1/2} e^{-s^2/2} ds$. Then $\Phi$ pushes forward the standard Gaussian
measure on $\RR$ to the uniform measure on the interval $[0,1]$. Moreover, $\Phi$ is clearly an $L$-Lipschitz function for $L = 1/\sqrt{2 \pi}$. Set,
$$ G(x_1,\ldots,x_n) = \left( \Phi(x_1),\ldots, \Phi(x_n) \right). $$
Then $G: \RR^n \rightarrow (0,1)^n$ is a homeomorphism, pushing $\gamma_n$ forward to $\lambda_n$, where $\lambda_n$ is the uniform measure on $(0,1)^n$. Moreover,
$G$ is  an $L$-Lipschitz function for $L = 1/\sqrt{2 \pi}$. In particular, denoting $h = f \circ G$,
\begin{equation} G(h^{-1}(t) + \eps B^n) \subseteq  f^{-1}(t) + \frac{\eps}{\sqrt{2 \pi}} B^n  \qquad \qquad (t \in \RR^\ell, \eps > 0). \label{eq_1727} \end{equation}
Since $G_*(\gamma_n) = \lambda_n$, it follows from (\ref{eq_1727}) that $$ \gamma_n \left \{ h^{-1}(t) + \eps B^n \right \}  \leq \lambda_n \left \{ f^{-1}(t) + (\eps / \sqrt{2 \pi}) B^n \right \}. $$
Applying Theorem
\ref{thm_gwaist} for the continuous function $h = f \circ G$, we find $t \in \RR^\ell$
such that
$$
 \frac{\lambda_n \left( f^{-1}(t) + \eps B^n \right)}{\beta_{\ell} \eps^{\ell}} \geq \frac{\gamma_n \left( h^{-1}(t) + \sqrt{2 \pi} \cdot \eps B^n \right)}{\beta_{\ell} \eps^{\ell}}
 \geq \frac{\gamma_{\ell}
\left(\sqrt{2 \pi} \cdot \eps  B^{\ell} \right)}{\beta_{\ell} \eps^{\ell}} \stackrel{\eps \rightarrow 0^+} \longrightarrow 1. $$
Therefore  $Vol_{n-\ell}^*(f^{-1}(t)) \geq 1$.
\end{proof}

\begin{corollary} Let $\ell \geq 1$ and let $0 < \lambda_1 \leq \lambda_2 \leq \ldots \leq \lambda_n$. Consider the box $Q = (0, \lambda_1) \times \ldots \times (0, \lambda_n)$.
Let $f: Q \rightarrow \RR^\ell$ be a real-analytic  map.
Then there exists $t \in \RR^\ell$ with $$ Vol_{n-\ell}(f^{-1}(t)) \geq \prod_{j = 1}^{n-\ell} \lambda_j. $$ \label{cor_932}
\end{corollary}

The estimate of Corollary~\ref{cor_932} is clearly tight, as shown by the example of linear functions.
Note that $f$ from Corollary~\ref{cor_932} is assumed to be real-analytic and not merely continuous.
Thus, any set of the form $f^{-1}(t)$ is a finite union of smooth manifolds, and in particular $Vol^*_{n-\ell}(f^{-1}(t))$
is equal to the $(n-\ell)$-dimensional Hausdorff measure of the set $f^{-1}(t)$,
denoted by $Vol_{n-\ell}(f^{-1}(t))$.

\begin{proof}[Proof of Corollary~\ref{cor_932}] Denote $T(x_1,\ldots,x_n) = (\lambda_1 x_1, \ldots, \lambda_n x_n)$.
It follows from the change-of-variables formula that  for any $(n-\ell)$-dimensional manifold $M \subseteq (0,1)^n$,
\begin{equation} Vol_{n-\ell}(T(M)) \geq Vol_{n-\ell}(M) \cdot \prod_{j = 1}^{n-\ell} \lambda_j. \label{eq_1707} \end{equation}
By Theorem~\ref{thm_1127_}, the function $h = f \circ T$ has a fiber $h^{-1}(t)$ with $Vol_{n-\ell}^*( h^{-1}(t) ) \geq 1$.
Thus the set $h^{-1}(t)$ is  a real-analytic variety of dimension at least $n-\ell$. The dimension of $f^{-1}(t) = T( h^{-1}(t) )$ equals to that of $h^{-1}(t)$. If this dimension is larger than $n-\ell$,
then the desired conclusion is trivial. Otherwise, there is a smooth $(n-\ell)$-dimensional manifold $M \subseteq h^{-1}(t)$ with
\begin{equation} Vol_{n-\ell}(M) = Vol_{n-\ell} (h^{-1}(t)) = Vol_{n-\ell}^* (h^{-1}(t)) \geq 1. \label{eq_1707_} \end{equation}
Since $T(M) \subseteq f^{-1}(t)$, the conclusion follows from (\ref{eq_1707}) and (\ref{eq_1707_}).
\end{proof}

We conjecture that the conclusion of Corollary~\ref{cor_932} holds true for all continuous maps $f: Q \rightarrow \RR$,
with $Vol_{n-\ell}$ replaced by $Vol_{n-\ell}^*$.
We move on to the proof of Theorem~\ref{thm_section}.
The following two lemmas are needed for the verification of the corresponding $(V, \ell, I)$-peak property.

\begin{lemma} Let $K, V \subseteq \RR^n$ be convex bodies such that $\int_V x dx  =0$.
Let  $\vphi$ be a log-concave probability density supported on $K$. Write $x_0 = \int_K x \vphi(x) dx$ for the barycenter of $\vphi$.
Then for all $r > 0$,
\begin{equation}  \frac{1}{r^n \cdot Vol_n(V)} \int_{K \cap (x_0 + r V)} \vphi \geq \frac{1}{Vol_n(K - r V)}. \label{eq_1730}
\end{equation}
\label{lem_1053_}
\end{lemma}

\begin{proof} The barycenter of $\tilde{V} = -V$ lies at the origin. The point $x_0 \in K$ is the barycenter of the probability density $\vphi$. Hence $x_0$ is also the barycenter of the probability density
$$ \vphi_r = \vphi * \frac{1_{r \tilde{V}}}{r^n \cdot Vol_n(V)}, $$
which is the convolution of $\vphi$ with the indicator function of the convex set $r \tilde{V}$, normalized to be a probability density.
Note that $\vphi_r(x_0)$ equals to the left-hand side of (\ref{eq_1730}).
The probability
density $\vphi_r$ is log-concave by Pr\'ekopa-Leindler and it is supported in $K - r V$. Now (\ref{eq_1730}) follows from an inequality of Spingarn \cite{spingarn},
according to which
\begin{equation}  \vphi_r(x_0) \geq \frac{1}{Vol_n(K - r V)} \qquad \qquad \qquad (r > 0). \label{eq_1753} \end{equation}
We cannot resist providing the standard short  proof of (\ref{eq_1753}): The function $g_r = 1/\vphi_r$ is convex
in $\Omega_r = \{ x \in K - r V \, ; \, \vphi_r(x) > 0 \}$, since $\vphi_r$ is log-concave. By Jensen's inequality
$$ g_r(x_0) = g_r \left( \int_{\Omega_r} x \vphi_r(x) dx \right) \leq \int_{\Omega_r} g_r(x) \vphi_r(x) dx = Vol_n(\Omega_r) \leq Vol_n(K - r V), $$
and (\ref{eq_1753}) is proven.
\end{proof}

\begin{lemma} Let $R > 0$ and assume that $K \subseteq R B^n$ is a convex body.
Let  $\nu$ be a probability measure supported on $K$ with a log-concave density. Write $x_0 = \int_K x d \nu(x)$.
Then for all $0 \leq r \leq 1$,
\begin{equation}  \nu (x_0 + r B^n)  \geq \frac{\beta_n \cdot r^n}{Vol_n(K) + C_{n,R} \cdot r}, \label{eq_1730_}
\end{equation}
where $C_{n,R} > 0 $ is a constant depending solely on $n$ and $R$. \label{lem_1053}
\end{lemma}

\begin{proof} For any convex set $K \subseteq R B^n$ and $t > 0$, by the monotonicity of mixed volumes (e.g., Schneider \cite{schneider})
we have the following inequality:
\begin{equation}
Vol_n(K + t B^n) - Vol_n(K) \leq Vol_n( R B^n + t B^n ) - Vol_n(R B^n). \label{eq_1706}
\end{equation}
To see this, expand $Vol_n(K + t B^n)$ as a polynomial in $t$, and observe that the coefficients
of this polynomial -- the intrinsic volumes -- are bounded by the corresponding coefficients
of the polynomial $Vol_n(R B^n + t B^n)$. Curiously, when $n=2$ and $t$ tends to zero, inequality (\ref{eq_1706})
amounts to the Archimedes postulate on convex curves.
It follows from (\ref{eq_1706}) that for a certain constant $C_{n, R} > 0$,
\begin{equation}  Vol_n(K + r B^n) \leq Vol_n(K) + C_{n,R} \cdot r \qquad \text{for all} \ \ 0 \leq r \leq 1. \label{eq_1714}
\end{equation}
Now (\ref{eq_1730_}) follows from (\ref{eq_1714}) and Lemma~\ref{lem_1053_}.
\end{proof}

\begin{proof}[Proof of Theorem~\ref{thm_section}]
Abbreviate $M =
\sup_{F \in AG_{n, \ell}} Vol_{\ell}(K \cap F)$.
Let $R > 0$ satisfy $K \subseteq R B^n$.
For $0 \leq r \leq 1$  set
$$ I(r) = \min \left \{ 1, \frac{\beta_\ell r^\ell}{M + C_{\ell,R} \cdot r} \right \}, $$
where $C_{\ell, R}$ is the constant from Lemma~\ref{lem_1053}.
For $r > 1$ we set $I(r) := I(1)$. Thus $I: [0, \infty) \rightarrow [0, 1]$ is a continuous function with
\begin{equation}  \lim_{r \rightarrow 0^+} \frac{I(r)}{\beta_\ell \cdot r^\ell} = \frac{1}{M}.
\label{eq_1332} \end{equation}
Assume that $f: K \rightarrow \RR$ is a continuous function.
Let $0 < \eps < 1$, and let $K_0$ be a convex
body contained in the interior of $K$ with
$$  Vol_n(K_0) \geq (1 - \eps) \cdot Vol_n(K). $$
Let $\vphi = 1_{K_0} / Vol_n(K_0)$ be the uniform probability density on the convex body $K_0$ and set $V = B^n$.
We would like to verify that $\vphi$ satisfies  the $(V, \ell, I)$-peak property.
Suppose that we are given an $\ell$-dimensional affine subspace $E \subseteq \RR^n$ and a log-concave function $\psi: E \rightarrow [0, +\infty)$
with $\int_E {\vphi \cdot \psi} = 1$. Write $\nu$ for the probability measure on $E$ with density $$ \vphi \cdot \psi \cdot 1_E = \frac{1_{K_0 \cap E} \cdot \psi}{Vol_n(K_0)}. $$
Then $\nu$ is a log-concave measure in the affine subspace $E$, which is in fact supported in $K_0 \cap E$.
According to Lemma~\ref{lem_1053}, for some $x_0 \in E$,
$$  \nu( x_0 + r B^n) \geq \frac{\beta_\ell r^\ell}{Vol_\ell(K_0 \cap E) + C_{\ell,R} \cdot r} \geq I(r)
\qquad \text{for all} \ 0 \leq r \leq 1. $$
In fact, $\nu(x_0 + r B^n) \geq \nu(x_0 + B^n) \geq I(1) = I(r)$ also for $r > 1$. We have thus verified the $(V, \ell, I)$-peak property of $\vphi$.
By the Tietze extension theorem, we may extend the given continuous function $f: K \rightarrow \RR^{\ell}$
to a continuous function $f: \RR^n \rightarrow \RR^{\ell}$. We may now apply Theorem~\ref{prop_941} and conclude that
for a certain point $t \in \RR^{\ell}$,
\begin{equation}  Vol_n \left \{ K_0 \cap ( f^{-1}(t) + r B^n ) \right \} \geq I(r) \cdot Vol_n(K_0) \qquad \qquad \text{for all} \ r > 0. \label{eq_1337}
\end{equation}
From (\ref{eq_1332}) and (\ref{eq_1337}),
\begin{align*}  & Vol_{n- \ell}^*  ( K \cap f^{-1}(t))  = \liminf_{r \rightarrow 0^+} \frac{Vol_n \left \{ (K \cap  f^{-1}(t)) + r B^n  \right \}}{\beta_{\ell} \cdot r^{\ell}}
\\ & \phantom{bla} \geq \liminf_{r \rightarrow 0^+} \frac{Vol_n \left \{ K_0 \cap  (f^{-1}(t) + r B^n)  \right \}}{\beta_{\ell} \cdot r^{\ell}}  \geq \frac{Vol_n(K_0)}{M } \geq (1 - \eps) \frac{Vol_n(K)}{M }.
\end{align*}
Since $\eps$ was arbitrary, we see that $\sup_{t \in \RR^\ell} Vol_{n-\ell}^*(K \cap f^{-1}(t)) \geq Vol_n(K) / M$, as desired.
\end{proof}

Recall that the $\ell$-waist of a convex body $K \subseteq \RR^n$ is defined via
$$ w_{\ell}(K) = \inf_{f: K \rightarrow \RR^{n-\ell}} \sup_{t \in \RR^{n-\ell}}
\left( Vol^*_{\ell} (f^{-1}(t) \right)^{1/\ell}, $$
where the infimum runs over all continuous functions  $f: K \rightarrow \RR^{n-\ell}$.

\begin{proposition} Let $K \subseteq \RR^n$ be a convex body, let $1 \leq k \leq n$ and let $E \subset \RR^n$
be a linear subspace with $\dim(E) = k$. Then,
\begin{enumerate}
\item[(i)] $\displaystyle w_{\ell}(K) \leq w_{\ell}(Proj_E K) $  for $\ell=1,\ldots,k$.
\item[(ii)] If $K = -K$   then $w_{\ell}(K) \leq 2 \cdot w_{\ell}(K \cap E)$ for $\ell=1,\ldots,k$.
\end{enumerate} \label{prop_1336}
\end{proposition}

\begin{proof} Begin with the proof of (i). Let $\eps > 0$. For any continuous function $f: Proj_E K \rightarrow \RR^{k - \ell}$ we set
$$ g(x) = \left(Proj_{E^{\perp}} x, f (Proj_E(x)) \right) \in
E^{\perp} \times \RR^{k-\ell} \cong \RR^{n-k} \times \RR^{k - \ell}. $$
Then $g: K \rightarrow \RR^{n-\ell}$ is continuous, and hence it has a large fiber $g^{-1}(t,s)$,
with $\ell$-dimensional Minkowski volume  at least $(w_{\ell}(K) - \eps)^{\ell}$.
This fiber is a subset of the $k$-dimensional affine subspace $(Proj_{E^{\perp}})^{-1}(t)$, and it is in fact contained in a translation of $f^{-1}(s)$. Therefore
$$ Vol_{\ell}^*(f^{-1}(s)) \geq Vol_{\ell}^*(g^{-1}(t,s)) \geq (w_{\ell}(K) - \eps)^{\ell}, $$ and conclusion (i) follows,
as $\eps > 0$ is arbitrary. We continue with the proof of (ii).
Abbreviate $P = Proj_{E^{\perp}}$.
For each boundary point $y \in \partial P(K)$ there is a point $ b_y \in K$ with
$$ P(b_y) = y. $$
By Michael's selection theorem, we may assume that $b: \partial P(K) \rightarrow K$ is continuous.
For $y \in P(K)$ let us set $b_y = \|y \| \cdot b_{y / \| y \|}$, where $\| \cdot \|$ is  the norm whose unit ball is $P(K)$.
Then $b: P(K) \rightarrow K$ is a continuous with $P(b_y) = y$ for all $y \in P(K)$.
We claim that for any $x \in K$,
\begin{equation}
x - b_{P(x)} \in \frac{K \cap E}{\lambda} \subseteq 2 (K \cap E),
\label{eq_947}
\end{equation}
where we set $\lambda = 1 / (1 + \| Px \|) \in [1/2,1]$.
Indeed, $x \in K$ and $-b_{Px}/\| Px \| \in K$, hence the point $z = \lambda(x - b_{Px}) = \lambda x - (1 - \lambda) b_{Px}/\|Px\|$ belongs to $K$.
However, $Pz = 0$ and therefore $z \in K \cap E$ and (\ref{eq_947}) follows.
Given a continuous function $f: K \cap E \rightarrow \RR^{k - \ell}$ let us denote
$$ g(x) = \left( P(x), f \left(  \frac{x - b_{Px}}{2} \right) \right) \in E^{\perp} \times \RR^{k-\ell} \cong \RR^{n-k} \times \RR^{k - \ell}. $$
The function $g: K \rightarrow \RR^{n- \ell}$ is a well-defined continuous function according to (\ref{eq_947}).
We continue as in the proof of (i): The function $g$ has a large fiber $g^{-1}(t,s)$, which is contained in a translation of
$2 f^{-1}(s)$. Hence for any $\eps > 0$ we find a fiber $f^{-1}(s)$ whose $\ell$-dimensional Minkowski volume is at least $(w_{\ell}(K) / 2 - \eps)^{\ell}$,
completing the proof of (ii).
\end{proof}

Suppose that $X$ is a finite-dimensional normed space and that $\mu$ is a log-concave probability measure supported
in its unit ball. The following theorem states that any continuous function $f: X \rightarrow \RR^{\ell}$ has a large fiber:

\begin{theorem} Let $K \subseteq \RR^n$ be a centrally-symmetric convex body, let $\ell=1,\ldots,n$,
let $\mu$ be a probability measure supported in $K$ with a log-concave density and let $f: K \rightarrow \RR^{\ell}$ be continuous.
Then there exists $t \in \RR^{\ell}$ such that
$$ \mu(f^{-1}(t)  + r K) \geq \left( \frac{r}{2+r} \right)^\ell \qquad \text{for all} \ 0 < r <  1. $$
\label{thm_1148}
\end{theorem}

\begin{proof} Set $V = K$ and $I(r) = r^{\ell} / (2 + r)^{\ell}$. Thanks to Theorem~\ref{prop_941}, all we need is to verify that the log-concave density $\vphi$
of $\mu$ has the $(V, \ell, I)$-peak property. Thus, let $E \subseteq \RR^n$ be an $\ell$-dimensional affine subspace,
and let $\psi: E \rightarrow [0, +\infty)$ be log-concave with $\int_E {\vphi \cdot \psi} = 1$. The affine subspace $E$ is a translate of a certain linear subspace  $F \subseteq \RR^n$.
The inclusion (\ref{eq_947}) proven above amounts to the existence of a certain point $b \in \RR^n$ with
\begin{equation}  K \cap E \subseteq b + 2 (K \cap F). \label{eq_1350} \end{equation}
 Write $\nu$ for the probability measure on $E$ with density $1_E \cdot \vphi \psi$,
a log-concave measure in $E$ supported in $K \cap E$. The convex set $K \cap F$ is centrally-symmetric, and in particular its barycenter lies at the origin. It follows
from Lemma~\ref{lem_1053_} that for a certain point $x_0 = \int x d\nu(x) \in E$ and for all $r > 0$,
\begin{equation}  \frac{1}{r^\ell \cdot Vol_{\ell}(K \cap F)} \cdot \nu \left( x_0 + r (K \cap F) \right) \geq \frac{1}{Vol_{\ell} \left( K \cap E - r (K \cap F) \right)}. \label{eq_1351}
\end{equation}
From (\ref{eq_1350}) and (\ref{eq_1351}), for all $r \geq 0$,
$$  \nu \left( x_0 + r (K \cap F) \right) \geq \frac{r^\ell \cdot Vol_{\ell}(K \cap F)}{Vol_{\ell} \left( b + 2(K \cap F) + r (K \cap F) \right)}   = I(r). \label{eq_1352}
$$
We have verified the $(V, \ell, I)$-peak property of $\vphi$. The conclusion  now follows from Theorem~\ref{prop_941}.
\end{proof}

Note that the ambient dimension $n$ does not appear in the estimate of Theorem \ref{thm_1148}. We therefore conjecture
that it is possible to formulate and prove an infinite-dimensional version of this theorem.

\section{The Gaussian $M$-position}
\setcounter{equation}{0}
\label{M_sec}

Theorem~\ref{thm_1752} and Theorem~\ref{thm_psitwo} are proven in this section.
We write
$c, C, \tilde{c}, c_1, C_2, \ldots$ for various positive universal constants, whose value is not
necessarily the same in different appearances.
Let $\mu$ be a probability measure on $\RR^n$ whose barycenter lies at the origin.
The $\psi_2$-constant of $\mu$ is the infimum over all $A > 0$ with the following property: For any linear functional $L: \RR^n \rightarrow \RR$,
\begin{equation}  \left( \int_{\RR^n} |L(x)|^p d \mu(x) \right)^{1/p} \leq A \sqrt{p} \cdot \int_{\RR^n} |L(x)| d \mu(x) \qquad \text{for all} \ p > 1.
\label{eq_926} \end{equation}
The covariance matrix of $\mu$, denoted by $Cov(\mu)$, is the matrix
whose $(i,j)$-entry is
$ \int_{\RR^n} x_i x_j d \mu(x)$.
Assume that $\mu$ has a log-concave density $\vphi$. The isotropic constant of $\mu$ is defined via $$ L_{\mu} = (\sup \vphi)^{1/n} \cdot ( \det Cov(\mu) )^{1/(2n)}. $$
See, e.g., the book by Brazitikos, Giannopoulos, Valettas and Vritisou \cite{Gian} for information about the isotropic constant.
In particular, it is shown in \cite[Proposition 2.3.12]{Gian} that $L_{\mu} > c$ for some universal constant $c > 0$.
Bourgain's slicing conjecture is equivalent to the hypothesis that $L_{\mu} < C$.
This conjecture was verified under $\psi_2$-assumptions by Bourgain (see \cite[Theorem 3.4.1]{Gian}).
The dependence on the $\psi_2$-constant was slightly improved by Dafnis and Paouris and by Klartag and Milman (see \cite[Theorem 7.5.15]{Gian}),
thus  when $\mu$ satisfies (\ref{eq_926}), we have the bound
$$  L_{\mu} < C \cdot A. $$
One says that the probability measure $\mu$ is {\it isotropic}, or that it is  in isotropic position, if its barycenter lies at the origin and $Cov(\mu)$ is a scalar matrix.
Recall that the $\psi_2$-constant of a convex body $K \subseteq \RR^n$ with barycenter at the origin
is defined to be the $\psi_2$-constant of $\mu_K$, the uniform probability measure on $K$.
The convex body $K$ is said to be in isotropic position if $\mu_K$ is.

\begin{lemma} Let $K \subseteq \RR^n$ be a convex body in isotropic position. Then for any $\ell=0,\ldots,n$
and an affine $\ell$-dimensional subspace $E \subseteq \RR^n$,
$$ Vol_{\ell} (K \cap E) \leq (C A)^{n-\ell} \cdot Vol_n(K)^{\ell/n}, $$
where $A$ is the $\psi_2$-constant of $K$ and $C > 0$ is a universal constant. \label{lem_1105}
\end{lemma}

\begin{proof}
Let $\lambda > 0$ be such that $Cov(\mu_K) = \lambda^2 \cdot Id$, where $Id$ is the identity matrix. Then,
\begin{equation}  L_{\mu_K} = Vol_n(K)^{-1/n} \cdot \det(Cov(\mu_K))^{1/(2n)} = Vol_n(K)^{-1/n} \cdot \lambda. \label{eq_1011}
\end{equation}
Let $F$ be the linear $\ell$-dimensional subspace which is a translate of $E$.
Define $\nu = (Proj_{F^{\perp}})_* \mu_K$. Then $\nu$ is a log-concave probability measure in $F^{\perp}$ with barycenter at the origin.
Moreover, the $\psi_2$-constant of $\nu$ is at most $A$. The covariance matrix of $\nu$ (or more precisely,
the covariance operator) equals $\lambda^2$ times the identity. Write $\vphi: F^{\perp} \rightarrow [0, \infty)$ for  the log-concave density of $\nu$.
Then,
\begin{equation}  L_{\nu} = (\sup \vphi)^{\frac{1}{n-\ell}} \cdot \det(Cov(\nu))^{\frac{1}{2(n-\ell)}} = \lambda \cdot (\sup \vphi)^{\frac{1}{n-\ell}} \geq \lambda \left( \frac{Vol_{\ell}(K \cap E)}{Vol_n(K)} \right)^{\frac{1}{n-\ell}}. \label{eq_1012}
\end{equation}
 From (\ref{eq_1011}) and (\ref{eq_1012}),
$$ (Vol_{\ell}(K \cap E))^{\frac{1}{n-\ell}} \leq \frac{L_{\nu}}{L_{\mu_K}} \cdot Vol_n(K)^{\frac{1}{n-\ell} - \frac{1}{n}}
\leq C A \cdot Vol_n(K)^{\frac{\ell}{n(n-\ell)}},
$$ where we used that $L_{\nu} < C \cdot A$ and $L_{\mu_K} > c$ in the last passage. \end{proof}

\begin{corollary} Let $K \subseteq \RR^n$ be a convex body in isotropic position. Write
$A > 0$ for the $\psi_2$-constant of $K$. Then for any $\ell=1,\ldots,n$ and a continuous function $f: K \rightarrow \RR^{\ell}$,
there exists $t \in \RR^{\ell}$ with
$$ Vol_{n-\ell}^*(f^{-1}(t)) \geq \left( \frac{c}{A} \cdot Vol_n(K)^{\frac{1}{n}} \right)^{n - \ell}, $$
where $c > 0$ is a universal constant. \label{cor_1716}
\end{corollary}

\begin{proof} By Theorem~\ref{thm_section} and Lemma~\ref{lem_1105},
\begin{equation*}
  \sup_{t \in \RR^{\ell}} Vol_{n-\ell}^*(f^{-1}(t)) \geq \frac{Vol_n(K)}{ (C A)^{n-\ell} \cdot Vol_n(K)^{\ell/n} } = \left( \frac{\tilde{c}}{A} \cdot Vol_n(K)^{1/n} \right)^{n - \ell}.
  \tag*{\qedhere}
\end{equation*}
\end{proof}

\medskip
A fundamental
component of high-dimensional convex geometry is Milman's  theorem on the existence of an $M$-ellipsoid.
We shall use the following version of this theorem, see Milman \cite{Mil} or Pisier's book \cite[Chapter 7]{Pis}:

\begin{theorem} Let $K \subseteq \RR^n$ be a centrally-symmetric convex body. Then there exists an ellipsoid
$\cE \subseteq \RR^n$ with $Vol_n(K) = Vol_n(\cE)$ and
$$ Vol_n(K \cap \cE) \geq c^n \cdot Vol_n(K), $$
where $c > 0$ is a universal constant. \label{thm_1008}
\end{theorem}

The ellipsoid $\cE$ from Theorem~\ref{thm_1008} is far from being a unique.
One possibility
for determining such an $M$-ellipsoid uniquely is to use a Gaussian minimization procedure.
This possibility is exploited in Bobkov \cite{bobkov},
where the following is proven (see also Rotem \cite[Remark 1]{rotem} for explanations regarding
the uniqueness part):

\begin{proposition} Let $K \subseteq \RR^n$ be a centrally-symmetric convex body of volume one.
Then there exists a unique symmetric, positive-definite, linear map $T_K \in SL_n(\RR)$
such that
\begin{equation} \gamma_n \left ( T_K( K)  \right ) = \sup_{T \in SL_n(\RR)} \gamma_n \left ( T(  K) \right ). \label{eq_1342}
\end{equation}
Moreover, let $\mu$ be the conditioning of $\gamma_n$ to $K$, i.e., $\mu(A) = \gamma_n(A \cap K) / \gamma_n(K)$ for all $A$.
Then the measure $\mu$ is isotropic. \label{prop_1400}
\end{proposition}

It is clear that the supremum on the
right-hand side of (\ref{eq_1342}) is attained also for $T = U T_K$, where
$U \in O(n)$ and $O(n)$ is the group of orthogonal transformations. The requirement that the linear map $T_K$ in Proposition~\ref{prop_1400} be symmetric and positive-definite 
seems  a bit artificial,
yet it breaks the symmetry and allows us to consistently select a uniquely-defined maximizer.
In the case of an arbitrary convex body $K \subseteq \RR^n$, not necessarily centrally-symmetric of volume one, we set
$$ T_K := T_{\alpha \cdot (K - b_K) \cap (b_K - K)} $$
where $b_K = \int_K x dx / Vol_n(K)$ is the barycenter of $K$ and $\alpha^{-n} = Vol_n \{ (K - b_K) \cap (b_K - K) \}$.
Recall the definition of $Symm(K)$
from the Introduction.

\begin{proposition} Let $K \subseteq \RR^n$ be a convex body with barycenter at the origin. Then $T_K T = T T_K$ for all $T \in Symm(K)$. \label{prop_1432}
\end{proposition}

\begin{proof} Note that $Symm(K) \subseteq O(n)$ since the barycenter of $K$ lies at the origin. For any $T \in Symm(K)$ we have $T(K) = K$ and also $T(-K) = - T(K) = -K$ and hence $T(K \cap (-K)) = K \cap -K$. The map $S = T^{-1} T_K T \in SL_n(\RR)$ is symmetric
and positive-definite,
with $$ \gamma_n \left \{ S (K \cap (-K)) \right \} = \gamma_n \left \{ T_K T  (K \cap (-K)) \right \} =\gamma_n \left \{ T_K  (K \cap (-K)) \right \}. $$
According to  the uniqueness part of Proposition~\ref{prop_1400}, necessarily $S = T_K$ and $T_K T = T T_K$.
\end{proof}

Let $K \subseteq \RR^n$ be a centrally-symmetric convex body of volume one.
We say that $K$ is in the {\it Gaussian $M$-position} if $T_K = Id$.

\begin{lemma} Let $K \subseteq \RR^n$ be a centrally-symmetric convex body of volume one.
Assume that $K$ is in the Gaussian $M$-position. Let $\mu$ be the conditioning of $\gamma_n$ to $K$. Then the $\psi_2$-constant of $\mu$
is at most $C$  and moreover $\gamma_n(K) \geq c^n$. \label{lem_1605}
\end{lemma}

\begin{proof}   By the ``Moreover'' part of Proposition~\ref{prop_1400},
\begin{equation}
Cov(\mu) = C_{\mu} \cdot Id, \label{eq_1633}
\end{equation}
where $Id$ is the identity matrix.
According to Bobkov \cite[Corollary 3.3]{bobkov},
\begin{equation}
\gamma_n(K) \geq c^n \cdot \sup_{\cE \subseteq \RR^n} Vol_n( K \cap \cE) \geq c_1^n \cdot Vol_n(K ) = c_1^n
\label{eq_1633_}
\end{equation}
where the supremum runs over all ellipsoids $\cE$ of volume one, and where the second inequality is the content of Theorem~\ref{thm_1008}.
Since $L_{\mu} > c$, from (\ref{eq_1633}) and (\ref{eq_1633_}),
\begin{equation}  C_{\mu} = L_{\mu}^2 \cdot \left( \frac{\gamma_n(K) }{ (2 \pi)^{-n/2} } \right)^{2/n} > c_2. \label{eq_849_} \end{equation}
By the comparison of moments of log-concave measures (e.g., \cite[Theorem 2.4.6]{Gian}) we deduce from (\ref{eq_1633}) and (\ref{eq_849_}) that
\begin{equation}
\int_{\RR^n} |\langle x, \theta \rangle| d \mu(x) \geq \tilde{c} \sqrt{ \int_{\RR^n} |\langle x, \theta \rangle|^2 d \mu(x) } = \tilde{c} \cdot \sqrt{C_{\mu}} \geq \bar{c}
\qquad (\theta \in S^{n-1}).
\label{eq_838}
\end{equation}
A  consequence of the Pr\'ekopa-Leindler inequality (see the proof of
Lemma 4 in Eldan and Lehec \cite{EL}) is that
$$ \int_{\RR^n} e^{\langle \xi, x \rangle} d \mu(x) \leq e^{|\xi|^2/2} \qquad \text{for} \ \xi \in \RR^n.
$$
Therefore
$$ \int_{\RR^n} e^{|\langle \xi, x \rangle|} d \mu(x) \leq \int_{\RR^n} e^{\langle \xi, x \rangle} d \mu(x)  + \int_{\RR^n} e^{-\langle \xi, x \rangle} d \mu(x)  \leq 2 e^{|\xi|^2/2}.
$$
Hence, for any $\theta \in S^{n-1}$ and an integer $p \geq 2$,
\begin{equation}
\int_{\RR^n} |\langle x, \theta \rangle|^p d \mu(x) \leq \frac{p!}{p^{p/2}} \int_{\RR^n} e^{\sqrt{p} |\langle x, \theta \rangle|} d \mu(x) \leq (C \sqrt{p} )^p \cdot
2 e^{p/2}, \label{eq_1629}
\end{equation}
where we used the inequality $(\sqrt{p} s)^p \leq p! \exp(\sqrt{p} s)$ for $s \geq 0$. The lemma now follows from (\ref{eq_838}) and (\ref{eq_1629}).
\end{proof}

An idea of K. Ball (see \cite[Section 2.5]{Gian}) is to represent the volume distribution of a log-concave
measure by a certain convex body.
Let $\rho: \RR^n \rightarrow [0, \infty)$ be an even, log-concave, probability density.
Write $\mu$ for the measure whose density is $\rho$ and define
$$  K(\mu) = \left \{ x \in \RR^n \, ; \, \int_0^\infty \vphi(rx) r^{n+1} dr \geq \frac{\vphi(0)}{n+2} \right \}. $$
Then $K(\mu)$ is a centrally-symmetric convex body (e.g., \cite[Theorem 2.5.5]{Gian} and see also \cite[Theorem 3.9]{MP} for the relation to the Busemann inequality).
If $\mu$ is isotropic, then also $K(\mu)$ is isotropic (e.g., \cite[Proposition 2.5.3(vi)]{Gian}). It is also known that
\begin{equation} 1 \leq \vphi(0) \cdot Vol_n(K(\mu)) \leq \frac{\left( (n+2)! \right)^{\frac{n}{n+2}}}{n!} \leq C. \label{eq_1249}
\end{equation}
Indeed, since $\vphi$ is log-concave and even, necessarily $\vphi(0) = \sup \vphi$.
Hence (\ref{eq_1249}) follows from \cite[Lemma 2.5.6 and Proposition 2.5.7(i)]{Gian}. Additionally, it follows from \cite[Lemma 2.5.2 and Proposition 2.5.3(iv)]{Gian} that
\begin{equation} K(\mu) \subseteq \overline{\{ x \in \RR^n \, ; \, \vphi(x) > 0 \}}. \label{eq_1701}
\end{equation}

\begin{lemma} Let $\mu$ be a probability measure on $\RR^n$ with an even, log-concave density $\vphi$.
Then the $\psi_2$-constant of $K(\mu)$ is at most $C$ times the $\psi_2$-constant of $\mu$.
\label{lem_1604}
\end{lemma}

\begin{proof} For $\theta \in S^{n-1}$ and $p \geq 0$ set $$ r_p(\theta) = \left( \frac{n+p}{\vphi(0)} \cdot \int_0^\infty \vphi(r \theta) r^{n+p-1} dr  \right)^{\frac{1}{n+p}}. $$
Note that $ K(\mu) = \{ r \theta \, ; \, \theta \in S^{n-1}, 0 \leq r \leq r_2(\theta) \}$.
For any unit vector $v \in S^{n-1}$ and $p \geq 0$ we integrate in polar coordinates and obtain
\begin{equation} \int_{\RR^n} |\langle x, v\rangle|^p \vphi(x) dx = \frac{\vphi(0)}{n+p} \cdot \int_{S^{n-1}} |\langle \theta, v \rangle|^p \cdot r_p(\theta)^{n+p} d \theta
\label{eq_1213}
\end{equation}
and
\begin{equation} \int_{K(\mu)} |\langle x, v\rangle|^p dx  = \frac{1}{n+p} \int_{S^{n-1}}
|\langle \theta, v \rangle|^p \cdot r_2(\theta)^{n+p} d \theta. \label{eq_1214_} \end{equation}
According to \cite[Lemma 2.2.4]{Gian}, the
function $p \mapsto r_p(\theta)$ is increasing in $p > 0$. Therefore, from (\ref{eq_1249}), (\ref{eq_1213}) and (\ref{eq_1214_}),
\begin{equation}  \frac{1}{Vol_n(K(\mu))} \int_{K(\mu)} |\langle x, v\rangle|^p dx \leq \int_{\RR^n} |\langle x, v\rangle|^p \vphi(x) dx \qquad \text{for} \ p \geq 2,
\label{eq_1254} \end{equation}
while
\begin{equation}  \frac{1}{Vol_n(K(\mu))} \int_{K(\mu)} |\langle x, v\rangle|^2 dx \geq \frac{1}{C} \cdot \int_{\RR^n} |\langle x, v\rangle|^2 \vphi(x) dx.
\label{eq_1255} \end{equation}
Write $A$ for the $\psi_2$-constant of $\mu$. Then by (\ref{eq_1254}) and (\ref{eq_1255}), for all $p \geq 2$,
\begin{align}  \nonumber \left(  \int_{K(\mu)} |\langle x, v\rangle|^p \frac{dx}{Vol_n(K(\mu))} \right)^{1/p} & \leq C A \sqrt{p} \left(  \int_{K(\mu)} |\langle x, v\rangle|^2 \frac{dx}{Vol_n(K(\mu))} \right)^{1/2} \\ & \leq \tilde{C} A \sqrt{p} \int_{K(\mu)} |\langle x, v\rangle| \frac{dx}{Vol_n(K(\mu))}, \label{eq_1602}
\end{align}
where we used comparison of moments (e.g., \cite[Theorem 2.4.6]{Gian}) in the last passage. According to (\ref{eq_1602}), the $\psi_2$-constant of $K(\mu)$ is at most $\tilde{C} A$.
\end{proof}

\begin{proposition} Let $K \subseteq \RR^n$ be a convex body with barycenter at the origin such that $K \cap (-K)$ is a convex body of volume one
in the Gaussian $M$-position.

\medskip Then there exists a centrally-symmetric, convex body $T \subseteq K$ in isotropic position whose $\psi_2$-constant is at most $C_1$,
such that
$$ \left( \frac{Vol_n(K)}{Vol_n(T)} \right)^{1/n} < C_2. $$
Here, $C_1, C_2 > 0$ are universal constants. \label{prop_1715}
\end{proposition}

\begin{proof} The Rogers-Shephard inequality (e.g., \cite[Theorem 4.1.20]{AGM}) states that
\begin{equation}  1 = Vol_n(K \cap (-K)) \geq 2^{-n} \cdot Vol_n(K). \label{eq_1703} \end{equation}
Write $\mu$ for the conditioning of $\gamma_n$ to the convex body $K \cap (-K)$, and let $\vphi$ be the log-concave probability density of $\mu$. According to Lemma
\ref{lem_1605},
\begin{equation}  \vphi(0) = \frac{(2 \pi)^{-n/2}}{\gamma_n(K \cap (-K))} \leq C^n. \label{eq_1655}
\end{equation}
By Proposition~\ref{prop_1400}, the measure $\mu$ is isotropic. According to Lemma~\ref{lem_1605}, the $\psi_2$-constant of $\mu$
is at most $C$. By Lemma~\ref{lem_1604}, also the $\psi_2$-constant of $T := K(\mu)$ is at most $C_1$ while $T$ is a centrally-symmetric convex body in isotropic position.
Additionally, since $\mu$ is supported in $K \cap (-K)$, we learn from (\ref{eq_1701}) that
$$ T \subseteq K \cap (-K) \subseteq K. $$
From (\ref{eq_1249})  and (\ref{eq_1655}) we deduce that $Vol_n(T) > c^{-n}$.
In view of (\ref{eq_1703}) we conclude that $Vol_n(K) \leq C_2^n \cdot Vol_n(T)$, completing the proof.
\end{proof}

Recall that if $K \subseteq \RR^n$ is a centrally-symmetric
convex body of volume one, then the body $T_K(K)$ is in the Gaussian $M$-position.

\begin{proof}[Proof of Theorem~\ref{thm_psitwo}] The $\psi_2$-constant of a convex body
is the same as the $\psi_2$-constant of its image under an invertible, linear transformation. Hence, in proving Theorem~\ref{thm_psitwo},
we may apply a linear map (just a dilation) and assume that
$$ Vol_n(K \cap (-K)) = 1. $$
Applying another linear map (the map $T_K$), we may further assume that $K \cap (-K)$ is in the Gaussian $M$-position.
The conclusion of the theorem now follows from Proposition~\ref{prop_1715}.
\end{proof}

\begin{proof}[Proof of Theorem~\ref{thm_1752}] The case $\ell = n$ is trivial, hence let us assume that $1 \leq \ell \leq n-1$.
Translating, we may also assume that the barycenter of $K$ lies at the origin.
Denote $\alpha = Vol_n(K \cap (-K))^{1/n}$. Then $\alpha \geq 1/2$ by the Rogers-Shephard inequality.
Setting $\tilde{K} = T_K(K)$, we obtain that
$$ K_1 = \frac{\tilde{K} \cap (-\tilde{K})}{\alpha} $$
is a centrally-symmetric convex body of volume one in the Gaussian $M$-position.
Apply 
Proposition~\ref{prop_1715} for $K_1$ to obtain the convex body $T \subseteq K_1$.
Since $T$ is in isotropic position with a $\psi_2$-constant bounded by $C$, according to 
Corollary~\ref{cor_1716},
$$ w_{\ell}(T) \geq  c \cdot Vol_n(T)^{\frac{1}{n}} \geq \tilde{c} \cdot Vol_n(K_1)^{\frac{1}{n}} = \tilde{c}. $$
Thanks to the homogeneity and monotonicity of waists,
$$ w_{\ell}(\tilde{K}) = \alpha \cdot w_{\ell}(\tilde{K} / \alpha) \geq \alpha \cdot w_{\ell}(K_1)  \geq \alpha \cdot w_{\ell}(T) \geq w_{\ell}(T) / 2 \geq \bar{c}, $$
completing the proof.
\end{proof}

\begin{corollary} Let $K \subseteq \RR^n$ be a convex body containing the origin in its interior.
Write $\| x \|_K = \inf \{ \lambda > 0 \, ; \, x / \lambda \in K \}$ for the Minkowski functional,
and $M(K) = \int_{S^{n-1}} \| x \|_K d \sigma(x)$, where $\sigma$ is the uniform probability measure on $S^{n-1}$. Then,
$$ w_{\ell}(K) \geq \frac{c}{\sqrt{n} \cdot M(K)} \qquad (\ell =1 ,\ldots, n), $$
where $c > 0$ is a universal constant.
\end{corollary}

\begin{proof} Denote $T = K \cap \left(\frac{1}{2M(K)} B^n \right)$.
By the Markov-Chebyshev inequality,
$$ \frac{Vol_n \left ( T \right )}{Vol_n \left ( \frac{1}{2M(K)} B^n\right )} \geq \sigma \left \{ x \in S^{n-1} \, ; \, \| x \|_K \leq 2 M(K) \right \} \geq \frac{1}{2}.
$$
According to Theorem~\ref{thm_section}, for $\ell=0,\ldots,n-1$,
\begin{align*}  w_{n-\ell}(K) & \geq w_{n-\ell}(T) \geq \left( \frac{Vol_n(T)}{\sup_{E \in AG_{n, \ell}} Vol_{\ell}(T \cap E) } \right)^{\frac{1}{n-\ell}}
\\ & \geq \left(  \frac{Vol_n \left ( \frac{1}{2M(K)} B^n\right )}{2 \cdot Vol_\ell \left ( \frac{1}{2M(K)} B^\ell \right ) } \right)^{\frac{1}{n-\ell}}
\geq \frac{c}{\sqrt{n} \cdot M(K)}. \tag*{\qedhere}
\end{align*}
\end{proof}

{
}

\bigskip
\noindent
Department of Mathematics, Weizmann Institute of Science, Rehovot 76100, Israel, and
School of Mathematical Sciences, Tel Aviv University, Tel Aviv 69978, Israel.  \\   e-mail:   \verb"klartagb@tau.ac.il"

\end{document}